\documentclass[11pt,a4paper]{article}

\usepackage{amssymb, amsmath, amsthm, amssymb}

\usepackage[cm]{fullpage}
\usepackage[english]{babel}
\usepackage[pdftex]{graphicx}
\usepackage{booktabs}
\usepackage{xcolor, enumerate}

\usepackage{hyperref}
\usepackage{autonum}
\usepackage{booktabs}

\numberwithin{equation}{section}

\newtheorem{thm}{Theorem}
\newtheorem{lem}{Lemma}
\newtheorem{prop}{Proposition}
\newtheorem{rem}{Remark}
\newtheorem{cor}{Corollary} 
\newtheorem{defi}{Definition} 

\newcommand{\veps}{\varepsilon}

\newcommand{\Le}{\mathcal{L}}
\newcommand{\R}{\mathbb{R}}
\newcommand{\N}{\mathbb{N}}
\newcommand{\Z}{\mathbb{Z}}
\newcommand{\dd}{\textup{d}}

\DeclareMathOperator*{\esssup}{ess\,sup}
\DeclareMathOperator*{\essinf}{ess\,inf}

\title{Numerical methods and regularity properties for viscosity solutions of nonlocal in space and time diffusion equations
}
\author{F\'elix del Teso\thanks{Departamento de Matem\'aticas, Universidad Aut\'onoma de Madrid, Spain, felix.delteso@uam.es}, \; \L ukasz P\l ociniczak\thanks{Faculty of Pure and Applied Mathematics, Wroc{\l}aw University of Science and Technology, Poland, lukasz.plociniczak@pwr.edu.pl}}
\date{}

\begin{document}
\maketitle

\begin{abstract}
We consider a  general family of  nonlocal in space and time diffusion equations with space-time dependent diffusivity and prove convergence of finite difference schemes in the context of viscosity solutions under very mild conditions. The proofs, based on regularity properties and compactness arguments on the numerical solution, allow to inherit a number of interesting results for the limit equation. More precisely, assuming H\"older regularity only on the initial condition, we prove convergence of the scheme, space-time H\"older regularity of the solution depending on the fractional orders of the operators, as well as specific blow up rates of the first time derivative. Finally, using the obtained regularity results, we are able to prove orders of convergence of the scheme in some cases. These results are consistent with previous studies. The schemes'  performance is further numerically verified using both constructed exact solutions and realistic examples. Our experiments show that multithreaded implementation yields an efficient method to solve nonlocal equations numerically. 
\end{abstract}

\textbf{Keywords: } anomalous diffusion, viscosity solution, Caputo derivative, nonlocal diffusion, L1 method, numerical scheme\\

\textbf{AMS Classification:} 35D40, 35R11, 65M06

\section{Introduction}
The goal of this paper is to propose convergent schemes in the context of viscosity solutions for the following class of space-time nonlocal equations with space-time dependent diffusivity,
\begin{equation}\label{eq:main}
\partial_t^\alpha u(x,t)= D(x,t)\Le^{\mu,\sigma} [u] (x,t), \quad x\in \R^d, \quad t\in(0,T),
\end{equation}
with  suitable continuous initial condition $u(x,0)=u_0(x)$. Now, we introduce every term in equation (\ref{eq:main}). 

\emph{Time derivative}. The operator $\partial_t^\alpha$ in \eqref{eq:main} is the Caputo derivative of order $\alpha\in(0,1)$ defined for a smooth function $\phi:\R_+\to\R$ as 
\begin{equation}\label{eq:capdef}
\partial_t^\alpha \phi(t)=\frac{1}{\Gamma(1-\alpha)}\frac{\phi(t)-\phi(0)}{t^{\alpha}}+\frac{\alpha}{\Gamma(1-\alpha)}\int_0^t \frac{\phi(t)-\phi(s)}{(t-s)^{1+\alpha}} \dd s.
\end{equation}
\begin{rem}
A simple integration by parts in \eqref{eq:capdef} recovers the usual definition of Caputo derivative given by
\[
\partial_t^\alpha \phi(t)= \frac{1}{\Gamma(1-\alpha)}\int_0^t \frac{\phi'(s)}{(t-s)^\alpha} \dd s.
\]
In this work we prefer to deal with \eqref{eq:capdef} since it naturally requires less regularity on $\phi$ in order to be well defined and it preserves the monotonicity properties that will be needed in the context of viscosity solutions.
\end{rem}
\emph{Diffusion operator.} The operator $\Le^{\mu,\sigma}$ belongs to the class of L\'evy type diffusion operators defined for smooth $\psi:\R^d\to \R$ as
\[
\Le^{\mu,\sigma}= L^\sigma + \Le^\mu,
\]
where the local and nonlocal parts are given by
\begin{equation}
L^\sigma[\psi](x):=\textup{tr}(\sigma \sigma^T D^2\phi(x))= \sum_{i=1}^P \partial_{\sigma_i}^2\psi(x), \quad \textup{where} \quad \partial_{\sigma_i}:=\sigma_i\cdot \nabla,
\end{equation}
\begin{equation}\label{eq:levydef2}
\Le^\mu [\psi](x):= \int_{|z|>0} \left( \psi(x+z)+\psi(x-z)-2\psi(x)\right)\dd \mu(z),
\end{equation}
where $\sigma=(\sigma_1,\ldots,\sigma_P)\in \R^{d\times P}$, $P\in \N$ and $\sigma_i\in \R^d$, and $\mu$ is a nonnegative symmetric Radon measure in $\R^d\setminus\{0\}$ satisfying 
\[
\int_{|z|>0}\min\{|z|^2,1\} \dd \mu(z)<+\infty.
\]
This class of diffusion operators coincides with the class of operators of symmetric Lévy processes. Some examples are the classical Laplacian $\Delta$, the fractional Laplacian $-(-\Delta)^s$ with order $s\in (0,1)$, relativistic Schr\"odinger-type operators $m^s Id - (m^2 Id - \Delta)^{s}$ with $m>0$, strongly degenerate operators like $\partial^2_{x_1}$ or $-(-\partial^2_{x_1})^s$, and even discrete operators \cite{dTEnJa17}.

\emph{Diffusivity coefficients.} The diffusivity is  a nonnegative function  $D: \R^d\times \R_+ \to \mathbb{R}_+$. More specific regularity assumptions will be given below in the statement of the results. 

We can give a physical motivation to consider the governing equation \eqref{eq:main}. For example, we can consider a porous medium wetted with some fluid. Moisture imbibition is driven by the flux arising from Darcy's law \cite{bear2013dynamics} while the nonlocal operators $\partial^\alpha_t$ and $\Le^{\mu,\sigma} [\psi](x)$ correspond to the anomalous diffusion effects \cite{metzler2000random}. As was shown in \cite{plociniczak2015analytical}, the emergence of the Caputo fractional derivative is related to the inclusion of the waiting-time phenomenon. According to that, the fluid's evolution may be slowed down as a result of chemical processes or the geometry of the porous medium \cite{El20}. On the other hand, some heterogeneous media exhibit the long-jump phenomena according to which the fluid can be transported for relatively long distances as a result of the complex geometry \cite{plociniczak2019derivation}.  When taking into account both waiting-time and long-jump phenomena and applying mass conservation for Darcy's law with the linear diffusivity, we can obtain the main equation governing the flow
\begin{equation}
\label{eq:Moisture}
\partial^\alpha_t v(x,t) = \nabla \cdot \left(D(x,t)\nabla^{2s-1}v(x,t)\right),
\end{equation}
where $v=v(x,t)$ is the moisture, $D=D(x,t)$ is the diffusivity of the medium, and $\nabla^{2s-1}$ is the fractional gradient operator related to the fractional Laplacian via $\nabla\cdot \nabla^{2s-1} = -(-\Delta)^s$. In the case $d=1$, there exists an attractive connection between \eqref{eq:Moisture} and the main equation \eqref{eq:main}. In realistic situations, the essentially one-dimensional flow frequently occurs in cases where the porous medium has one dimension substantially larger than the others (see \cite{kuntz2001experimental}). To obtain, at least formally, the governing equation \eqref{eq:main}, we can define the cumulative mass function
\[
	u(x,t) = \int_{-\infty}^x v(r,t) \dd r,
\]
where we assume that $v$ has finite mass. Now, integrating \eqref{eq:Moisture} yields
\[
	\partial^\alpha_t u(x,t) = D(x,t) \nabla^{2s-1} v(x,t) = -D(x,t) (-\Delta)^s u(x,t), 
\]
which is \eqref{eq:main} in the case when the nonlocal operator is the fractional Laplacian. In \cite{BiKaMo10,StTeVa16,StdTVa18}, a local in time fractional porous medium equation with nonlocal pressure was investigated with a similar approach based on an integrated problem. Later in \cite{del2023convergent} it was shown that the original solution can be numerically approximated by reconstructing it from the numerical scheme for the integrated variable. As in this paper, the convergence obtained in the integrated variable is locally uniform, while in the original one it has to be understood in the sense of the Rubinstein-Kantorovich distance.

Apart from hydrological applications, nonlocal operators are used to model various memory effects (temporal nonlocality) or long-range forces (spatial nonlocality). One of the most profound examples is the anomalous diffusion of a randomly walking particle for which the mean-squared displacement deviates from the linear in time relation \cite{metzler2000random}. The slower than classical diffusion is called the \textit{subdiffusion} and can frequently be modeled by inclusion of the Caputo derivative in the governing PDE.  Some specific examples include percolation of moisture into building materials (see \cite{chen2013variable} and references therein) and solute transport in soil \cite{pachepsky2000simulating}. A successful use of fractional derivatives in modeling of cell rheology has been conducted in \cite{djordjevic2003fractional}. Actually, one of the first examples of modeling with these nonlocal in time operators has been conducted in viscoelasticity \cite{bagley1983theoretical}. On the other hand, the faster regime is called the \textit{superdiffusion} which is frequently modeled by nonlocal in space operators such as the fractional Laplacian. Examples include protozoa migration \cite{alves2016transient}, single particle tracking in biophysics \cite{Sun17, wong2004anomalous}, plasma physics \cite{Del05}, astrophysics \cite{lawrence1993anomalous}, and financial mathematics \cite{jacquier2020anomalous}. Also, in a very natural way space nonlocal operators are central mathematical objects used in peridynamics, the theory of deformations with discontinuities \cite{diehl2019review}. 

\subsection{Comments on the related results}

In the literature, there are several related studies concerning various diffusion equations with the Caputo derivative. First of all, in \cite{topp2017existence}, the authors considered a general nonlinear diffusion problem and their existence and uniqueness results in the viscosity setting can be applied to our problem \eqref{eq:main} provided the diffusivity is only space dependent, that is, when $D=D(x)$. On the other hand, in \cite{namba2018existence}, the authors developed sufficient conditions for a unique viscosity solution to exist in the case of general nonlinear, but local, diffusion. They used Perron's method and a variant of the comparison principle. An approach was taken in \cite{giga2020discrete} to prove the existence of a viscosity solution through a discrete scheme. A similar technique is also utilized in the present paper. On the other hand, several authors considered weak solutions to PDEs that are related to ours. For example, results about existence and H\"older regularity were obtained in \cite{allen2016parabolic}. However, the uniqueness theorem required time-independence of the coefficient. This assumption was later relaxed in \cite{allen2017uniqueness} in the local in space setting. Another approach to weak solutions with a very general nonlocal in time operator and local diffusion was considered in \cite{wittbold2021bounded}. A generalization to nonlinear equations, such as the time-fractional porous medium equation, was carried out in \cite{akagi2019fractional} (see also \cite{plociniczak2018existence} for a different treatment of self-similar solutions). Apart from these well-posedness results, the long-time behavior of various solutions has been studied, for example, in \cite{dipierro2019decay}, where the spatial operator is allowed to be a general nonlinear local or nonlocal one (see also \cite{vergara2015optimal} for local diffusion but general nonlocal-in-time derivative).

As will be described in the following, we construct our numerical scheme based on the so-called L1 discretization of the Caputo derivative. It was introduced in \cite{oldham1974fractional} to discretize the Riemann-Liouville fractional derivative by means of a piecewise linear approximation. For sufficiently smooth functions, say at least $C^2([0,T])$, the method attains an order $2-\alpha$ (see, for example, \cite{li2019theory}, Theorem 4.1). The best error constant in this case was found in \cite{plociniczak2023linear}. When the time regularity of the differentiated function is lower, the order deteriorates. Typically, solutions to the time-fractional diffusion equation are H\"older continuous in time (see \cite{allen2016parabolic}). For example, it has been proved in \cite{sakamoto2011initial} that, for a local diffusion with constant diffusivity and with the Caputo derivative, equation \eqref{eq:main} is $\alpha$-H\"older continuous. In the present paper, we prove this result for the full generality of our main PDE: nonlocal in space and time along with space and time dependent diffusivity. In that case, the L1 method is globally only of order $\alpha$, what has been shown in \cite{stynes2017error,kopteva2019error} assuming the existence of the second derivative. In the following, we are able to relax this restrictive assumption and prove that for $\alpha\in [1/2,1)$ the global order is always $1-\alpha$ while for $\alpha\in(0,1/2)$ it deteriorates to $\alpha$. However, in the latter case the pointwise order still remains equal to $1-\alpha$. Therefore, our global estimate is the same as in the literature for $\alpha\in(0,1/2)$ but obtained under weaker assumptions. 

The L1 method is very versatile when it comes to discretizing the Caputo derivative. It possesses good monotonicity properties and a variety of Gr\"onwall inequalities (see for example Lemma 2.2 in \cite{liao2018sharp}). However, the latter is usually proved via very technical considerations. In the Appendix \ref{app:Gronwall} we prove a unified Caputo and Riemann-Liouville discrete fractional Gr\"onwall inequality using only elementary techniques and a classical result on integral equations. The literature on applying the L1 method to various subdiffusion problems is vast. From the related results, we mention only those concerning diffusion local in space since, according to our knowledge, the present paper is the first one considering both space and time nonlocal diffusion in the L1 numerical setting. In \cite{kopteva2019error}, the author conducts a thorough analysis of a time-fractional diffusion in which the diffusivity can depend on space. Particularly, it is shown that by choosing a graded mesh one can re-obtain the optimal order of convergence in time, that is $2-\alpha$ for both finite element and finite difference methods in space with nonsmooth initial data. An interesting analysis of the constant diffusivity semilinear case was conduced in \cite{al2019numerical} with operator theory techniques. There, it has been shown that the solution is H\"older continuous and the numerical method has an optimal error estimate. The case of numerical methods for local in space time-fractional diffusion with time-dependent diffusivity is still under vigorous research even in the semi-discrete case of only space discretization. For example, assuming that the initial data can be nonsmooth, the authors of \cite{jin2019subdiffusion,jin2020subdiffusion} prove optimal error bounds for this problem with Galerkin space discretization. The technique here is the perturbation of the purely space-dependent diffusivity equation. In \cite{mustapha2018fem}, a different approach is taken based on an energy method to obtain optimal error estimates. These works, however, assume some very stringent regularity properties on time diffusivity which may seem unnatural in typical situations. This problem has been solved in \cite{plociniczak2022error} using a refinement of the energy method and was extended to the semilinear equation. There, it is required that the diffusivity is only absolutely continuous with respect to time. This approach, although with regular initial data, was also carried out in the quasi-linear time-fractional diffusion equation in \cite{plociniczak2023linear}. A thorough overview of the L1 method applied to the time-fractional diffusion equation with nonsmooth initial data was published in \cite{jin2019numerical}. In the case of self-similar solutions to the time-fractional porous medium equations, there exists an efficient numerical method based on integral equations \cite{plociniczak2019numerical,plociniczak2014approximation}. 

As for numerical methods for local-in-time equations involving general L\'evy type diffusion operators we mention \cite{Dro10,DrJa14,CiJa14,dTEnJa18,dTEnJa19}. A general overview of L\'evy processes can be found in \cite{App09}.

\subsection{Comments on the main results}

The main contributions of this paper can be summarized as follows. See Section \ref{sec:MainResults} for a technical r\'esum\'e. 
\begin{itemize}
\item (\textit{Numerical scheme}) We develop a \textit{general framework} to demonstrate the convergence of a numerical scheme to solve \eqref{eq:main} in the viscosity sense. We allow for an \textit{arbitrary} discretization of the spatial operator, while the time marching is implemented with the L1 method. The convergence of the scheme is proved by showing uniform space-time regularity estimates and using Arzel\`a-Ascoli theorem to ensure the existence of the limit. We further demonstrate that the continuous solution can be understood in the viscosity setting. The only assumption we make is the requirement of H\"older continuity of the \textit{initial data} and not the solution itself. The resulting regularity of the viscosity limit is governed by the orders of differential operators and the regularity of the initial condition. 

\item (\textit{Existence of viscosity solutions}) The numerical scheme converges to a continuous limit and we prove that it is itself the \textit{viscosity solution} of \eqref{eq:main}. In some cases, and under some further assumptions on the discretization of the nonlocal-in-space operator and the regularity of the initial data, we show that the solution can also be understood in the pointwise sense. 
\item (\textit{Regularity of viscosity solutions}) The regularity of the continuous solution is inherited from the discrete level. Specifically, assume that the order of the spatial differential operator is $r\in [0,2]$ and the initial data is $a$-H\"older continuous with $a\in (0,r]$. Then, we prove that the viscosity solution is $\min\{a, a/r\}$-H\"older in space and $\alpha a / r$-H\"older in time. In the following, we state an even more general result for the moduli of continuity. The main point is that our only assumptions concern the \textit{initial data}, which is in contrast with many papers in the literature, where it is usually assumed that the solution to the main PDE possesses a certain smoothness. We prove it \textit{a priori}. We also prove precise blow up rates of the first time derivative at time $t=0$, which are in accordance to previous results in the literature.
\item (\textit{Errors of the scheme}) Assuming only H\"older regularity of the initial data, we are able to find the orders of convergence of the devised numerical scheme. Using the notation of above, and provided that $r\in [0,1)$ and $a\in (r,1]$ we show that the uniform order of convergence in space is $a-r$. Moreover, with a careful and novel analysis of the L1 scheme, we are able to prove that for $\alpha\in [1/2,1)$ the \emph{global} order in time is $1-\alpha$ while for $\alpha\in(0,1/2)$ it becomes $\alpha$. However, in the latter case the \emph{pointwise} order is larger and equal to $1-\alpha$.  For $\alpha\in (0,1/2)$, this result is consistent with the literature for local in space subdiffusion; however, in previous studies, it is generally assumed that the solution to the  PDE is $C^2((0,T])$ in time and $\|\partial^m_t u(t)\| \leq C t^{\alpha-m}$ with $m=0,1,2$. In this paper, we are able to completely \textit{relax this stringent assumption} requiring nothing from the solution itself.  
\end{itemize}

As a technical result stated in the Appendix \ref{app:Gronwall}, we give \textit{an elementary proof} of a discrete fractional Gr\"onwall inequality unified for both the Riemann-Liouville and Caputo derivatives. We prove a very general version of this result that might be utilized in future works within the convergence analysis for nonlinear equations. For the present paper, the main application of the Gr\"onwall inequality lies in the proof of the time regularity of the scheme. According to the strategy of our framework, this regularity is then inherited by the viscosity solution of the main PDE \eqref{eq:main}.  

\subsection{Notation and definitions}
Here, we briefly establish the notation for various function spaces and other objects that will be used in this paper.
\begin{itemize}
\item $Q_T=\R^d\times[0,T]$ is the closure of the domain related to the main equation \eqref{eq:main}.
\item $L^p(\mathbb{R}^d)\equiv$ Lebesgue functions in $\R^d$ with $1\leq p\leq \infty$.
\item $\mathcal{B}(\R^d)\equiv$ Bounded functions on $\R^d$ with pointwise values everywhere in $\R^d$.
\item $BUC(\R^d) \equiv$ Bounded uniformly continuous functions on $\R^d$.
\item $X^r_b(\Omega) \equiv$ Bounded H\"older continuous functions for $r\in(0,2]$ defined on $\Omega \subseteq \mathbb{R}^d$ with the norm
\[
\|\phi\|_{X^r(\Omega)}=\|\phi\|_{L^\infty(\Omega)} + \left\{ \begin{split}
	&\sup_{x,y\in \Omega,\, x\not=y} \frac{|\phi(x)-\phi(y)|}{|x-y|^r}, \hspace{5.35cm} \textup{if}  \quad r\in(0,1],\\
	&\sup_{x,y\in \Omega,\, x\not=y} \frac{|\phi(x)-\phi(y)|}{|x-y|}+ \sup_{x,y\in \Omega,\, x\not=y} \frac{|\nabla\phi(x)-\nabla\phi(y)|}{|x-y|^{r-1}}, \quad \textup{if} \quad r\in(1,2].
\end{split} \right.
\]  
\item $E_{\alpha,\beta}(z)\equiv$ Mittag-Leffler function for $\alpha>0$, $\beta > 0$, and $z\in\mathbb{C}$ is defined as follows 
\begin{equation}\label{eq:Mittag-Leffler}
	E_{\alpha,\beta}(z) = \sum_{k=0}^\infty \frac{z^k}{\Gamma(\alpha k + \beta)}, \quad E_\alpha := E_{\alpha,1}.  
\end{equation}
\end{itemize}

\section{Scheme, assumptions, and main results}
In this section, we introduce discretizations of the operators involved and the scheme, discuss the main assumptions, and formulate the main results of the paper.

\emph{Discretization of the time derivative}. We consider a uniform time grid parametrized by a time step $\tau>0$. More precisely, define $t_n:=\tau n \in \tau \N$. For a function $v: \tau \N\to\R$, we denote $v^n:=v(\tau n)$.
We introduce now the discretization of the Caputo derivative, for which we use the so-called L1 scheme given by
\begin{equation}\label{def:disccap}
\partial_\tau^\alpha v^n:= \frac{\tau^{-\alpha}}{\Gamma(2-\alpha)}\left( v^n - b_{n-1}v^0-\sum_{k=1}^{n-1}(b_{n-k-1}-b_{n-k})v^k\right), \quad \textup{where} \quad b_i:=(i+1)^{1-\alpha}-i^{1-\alpha}.
\end{equation}
\emph{Discretization of the spatial diffusion operator}.  We also consider the space discretization parameter $h>0$ and $z_\beta:=h\beta\in h\Z^d$. For convenience, we defined the class of discretization operators in a general form for functions $\psi:\R^d\to\R$ given by
\begin{equation}\label{eq:SpaceOperatorDiscrete}
L_h[\psi](x)=\sum_{\beta\not=0} \left(\psi(x+z_\beta)+\psi(x-z_\beta)-2\psi(x) \right)\omega_\beta(h).
\end{equation}
\begin{rem}
Note that $L_h$ can be defined for the function $\psi:h\Z^d\to \R$ (denote $\psi_\gamma=\psi(h\gamma)$ for $\gamma\in \Z^d$) as
\[
L_h[\psi]_\gamma= \sum_{\beta\not=0} \left(\psi_{\gamma+\beta}+\psi_{\gamma-\beta}-2\psi_\gamma \right)\omega_\beta(h).
\]
\end{rem}

In order to ensure good properties of the numerical scheme, we need to assume some monotonicity plus boundedness conditions on the weights. First, we need symmetry plus positivity of the weights:
\begin{equation}\label{As:wepos}\tag{$\textup{A1}_\omega$}
w_\beta(h)=w_{-\beta}(h)\geq0, \quad \textup{for all} \quad \beta \in\Z^d, \quad \textup{and} \quad   \sum_{\beta\not=0} \omega_\beta(h)<+\infty. 
\end{equation}
We also need to assume the following uniform boundedness plus moment condition:
\begin{equation}\label{As:webdd}\tag{$\textup{A2}_\omega$}
\textup{There exist $r\in[0,2]$ and $C_{r}>0$ such that} \,  \sup_{h\in(0,1)}\sum_{\beta\not=0} \min\{|\beta h|^r,1\}\omega_\beta(h)<C_{r}.
\end{equation}
As it can be seen in \cite{dTEnJa18,del2023convergent}, this is a natural assumption satisfied for any monotone finite difference discretization present in the literature. In the limit, the above is a condition on the measure $\mu$ in $\Le^\mu$ given in \eqref{eq:levydef2}
\begin{equation}\label{As:Le}\tag{$\textup{A}_\mu$}
\textup{There exist $r\in[0,2]$ and $C_{r}>0$ such that} \,  \int_{|z|>0} \min\{|z|^r,1\}\dd \mu(z)<C_{r}.
\end{equation}

\begin{rem}\label{rem:boundwei}
Note that assumption \eqref{As:wepos} ensures the boundedness of $L_h[\phi]$ for $\phi  \in \mathcal{B}(\R^d)$, i.e.,
\[
\|L_h[\phi]\|_{L^\infty(\R^d)} \leq  4 \| \phi\|_{L^\infty(\R^d)} \sum_{\beta\not= 0}\omega_\beta(h)<+\infty.
\]
Moreover, assumption \eqref{As:webdd} ensures the uniform boundedness in $h$ of $L_h[\phi]$ for $\phi  \in X^r_b(\R^d)$. More precisely,
\[
\begin{split}
	\|L_h[\phi]\|_{L^\infty(\R^d)} &\leq \|\phi\|_{X^r(\R^d)} \sum_{0<|\beta h|<1}|\beta h|^r \omega_\beta(h) + 4 \| \phi\|_{L^\infty(\R^d)} \sum_{|\beta h|\geq 1}\omega_\beta(h)\\
	&\leq C_r (\| \phi\|_{X^r(\R^d)} +4 \| \phi\|_{L^\infty(\R^d)}  ).
\end{split}
\]
\end{rem}

Since we are dealing with generic discretizations of the spatial operator, we need to assume some consistency on the discretization. In particular, we will use the following mild assumption in order to ensure convergence of the numerical scheme in the context of viscosity solutions:
\begin{equation}\label{As:conweak}\tag{$\textup{A1}_{\textup{c}}$}
\textup{For $\phi\in C^2_b(\R^d)$, we have $\|L_h \phi - \Le^{\mu,\sigma}\phi\|_{L^\infty(\R^d)}\leq  \|\phi\|_{C^2_b(\R^d)} o_h(1)$}, \quad \text{as} \quad h\rightarrow 0^+.
\end{equation}
Moreover, under the following more restrictive assumption, we are able to prove the convergence (with orders) of the numerical scheme to classical solutions. Note that this will only work when $r<1$ in assumption \eqref{As:webdd}, that is, when the order of differentiability of the operator is strictly less than one (for example, the fractional Laplacian $-(-\Delta)^s$ for $s<1/2$). More precisely, we assume that
\begin{equation}\label{As:constrong}\tag{$\textup{A2}_{\textup{c}}$}
\textup{If \eqref{As:webdd} holds for $r<1$ and $\phi\in X^a_b(\R^d)$ for  $a\in(r,1]$, then $\|L_h \phi - \Le^{\mu,\sigma}\phi\|_{L^\infty(\R^d)}\leq  \|\phi\|_{X^a_b(\R^d)} O(h^{a-r})$},
\end{equation}
as $h\rightarrow 0^+$. Finally, we denote $D^n(x):= D(x,t_n)$ and $D^n_\beta=D(x_\beta,t_n)$.

\emph{Formulation of the scheme}.  We are now ready to define the numerical scheme associated to our problem. Given an initial condition $U^0=u_0\in C_b(\R^d)$, we define iteratively for $n\geq1$ the function
\begin{equation}\label{eq:NumSch}\tag{S}
\partial^\alpha_\tau U^n(x)= D^n(x) L_h U^n(x), \quad x\in \R^d, \quad t_n:=n\tau \in(0,T].
\end{equation}

\begin{rem}
Note that we can restrict the above scheme to a uniform grid in space to deal with a fully discrete problem of the form 
\[
\partial^\alpha_\tau U^n_\beta= D^n_\beta L_h U^n_\beta, \quad \beta \in \Z^d, \quad t_n:=n\tau \in(0,T].
\]
With this in mind, any result of this paper can be formulated in a continuous or discrete in space form. We will prove them in continuous form for notational convenience.
\end{rem}

\subsection{Main results}\label{sec:MainResults}
We state now our main results. To shorten the presentation, we do not include the most general assumptions here. However, in Sections \ref{sec:PropNumSch} and \ref{sec:CompConv}, the results are stated and proved step by step only using the necessary assumptions for the proofs.

First, we state the result of well-posedness and properties of the numerical scheme.

\begin{thm}\label{thm:MainResultsScheme}
Assume \eqref{As:wepos}, \eqref{As:webdd} for some $r\in(0,2]$ and $h,\tau>0$. Let $u_0\in X^a_b(\R^d)$ for some $a\in(0,r]$, and $D\in C_b([0,T]: X^r_b(\R^d)) \cap C_b(\R^d: Lip([0,T]))$ nonnegative.
\begin{enumerate}[\rm (a)]
	\item \emph{(Existence and uniqueness)} There exists a unique solution $U^n$ of \eqref{eq:NumSch} with $U^0=u_0$.
	\item \emph{($L^\infty$-Stability)} $\|U^n\|_{L^\infty(\R^d)} \leq \|u_0\|_{L^\infty(\R^d)}$ for all $n\geq0$.
	\item \emph{($L^\infty$-Contraction)} Let $V^n$ satisfy \eqref{eq:NumSch} with $V^0=v_0\in X^a_b(\R^d)$. Then, for all $n\geq0$, we have \[\|U^n-V^n\|_{L^\infty(\R^d)}\leq \|u_0-v_0\|_{L^\infty(\R^d)}.\]
	\item \emph{(Space equicontinuity)} There exists  $C=C(\alpha,r, R,D, u_0)>0$ such that
	\[
	\|U^n(\cdot+y)-U^n\|_{L^\infty(\R^d)} \leq C \max\{1, t_{n}^\alpha\} \left\{ \begin{split} \max\{|y|^a, |y|^r\}, \quad  & \textup{if} \quad r\in (0,1],\\ \max\{|y|^\frac{a}{r},|y|\}, \quad  & \textup{if} \quad r\in (1,2]. \quad
	\end{split}\right.
	\]
	\item \emph{(Time equicontinuity)} There exists $C=C(\alpha,r, R,D, u_0)>0$ such that, for all $0\leq t_m< t_n$, we have
	\[
	\|U^n - U^m\|_{L^\infty(\mathbb{R}^d)} \leq C \max\{(t_n-t_m)^{\alpha\frac{a}{r}}, (t_n-t_m)^{\frac{a}{r}}\},
	\]
	and 
	\[
	\|U^n - U^m\|_{L^\infty(\mathbb{R}^d)} \leq  C ((t_n^{\alpha-1}+1)(t_n-t_m) )^\frac{a}{r}.
	\]
\end{enumerate}

\end{thm}

Next, with the use of uniform regularity estimates and the Arzel\`a-Ascoli theorem we will prove that the numerical scheme converges to a limit. The result is the following. 

\begin{cor}\label{thm:MainResultsConvergence}
Let the assumptions of Theorem \ref{thm:MainResultsScheme} hold. There exists a function $u\in C_b(Q_T)$ and a subsequence $(h_j,\tau_j)\to0$ as $j\to \infty$ such that
\[
	U\to u \quad \textup{locally uniformly in} \quad \mathbb{R}^d\times[0,T] \quad \textup{as} \quad j\to \infty. 
\]
\end{cor}

The obtained continuous solution is itself a viscosity solution of \eqref{eq:main} that inherits the regularity of the discrete scheme. At least in several special cases, the literature gives sufficient conditions for the uniqueness of such solutions. A more detailed discussion is given below in Remark \ref{rem:Uniqueness}.  The main result is the following. 

\begin{thm}\label{thm:MainResultsViscosity}
Let the assumptions of Theorem \ref{thm:MainResultsScheme} hold. Then the function $u$ given in Corollary \ref{thm:MainResultsConvergence} is a viscosity solution of \eqref{eq:main}. Furthermore, 
\begin{enumerate}[\rm (a)]
	\item \emph{(Convergence)} If viscosity solutions of \eqref{eq:main} are unique, the numerical scheme converges, that is,
	\[
		U\to u \quad \textup{locally uniformly in} \quad \R^d\times[0,T] \quad \textup{as} \quad (h,\tau)\to 0^+. 
	\]
	\item \emph{(Boundedness)} $
	\|u(\cdot,t)\|_{L^\infty(\R^d)} \leq \|u_0\|_{L^\infty(\R^d)}$  \textup{for all} $ t\in[0,T]$. 
	\item \emph{(Space equicontinuity)} There exists $C=C(\alpha,r, T, D, u_0)>0$ such that
	\[
	\|u(\cdot+y,t)-u(\cdot,t)\|_{L^\infty(\R^d)} \leq C\max\{1, t^\alpha\} \left\{ \begin{split}
		 \max\{|y|^a, |y|^r\},\quad  & \textup{if} \quad r\in (0,1],\\
		 \max\{|y|^\frac{a}{r},|y|\}, \quad  & \textup{if} \quad r\in (1,2], \quad
	\end{split}\right. \quad \textup{for all} \quad t\in[0,T].
	\]
	\item \emph{(Time equicontinuity)} There exists $C=C(\alpha,r, T, D, u_0)>0$ such that, for all $0\leq t<\tilde{t}$, we have
	\begin{equation}
		\|u(\cdot,\tilde{t}) - u(\cdot,t)\|_{L^\infty(\mathbb{R}^d)} \leq C \max\{(\tilde{t}-t)^{\alpha\frac{a}{r}}, (\tilde{t}-t)^{\frac{a}{r}}\}
	\end{equation}	
	and 
	\begin{equation}
		\|u(\cdot,\tilde{t}) - u(\cdot,t)\|_{L^\infty(\mathbb{R}^d)} \leq  C ((\tilde{t}^{\alpha-1}+1)(\tilde{t}-t) )^\frac{a}{r}.
	\end{equation}
	\item \label{item-e-limit} \emph{(Classical solution)} Additionally, assume \eqref{As:constrong}, \eqref{As:Le} for some $r\in[0,1)$, and let $u_0\in X^a_b(\R^d)$ with $a\in(r,1]$. Then, the limit $u$ satisfies \eqref{eq:main} in the pointwise sense. 
\end{enumerate}
\end{thm}

Recall that $X^c_b \subseteq X^a_b$ for all $c \in (0,a]$. From here, we see that the initial condition in Theorem \ref{thm:MainResultsViscosity} \eqref{item-e-limit} belongs to $X^c_b$ with $c \in (0,a]$. Therefore, all previous assertions of this theorem are also satisfied for classical solutions.

Note that $L^\infty$-contraction of Theorem \ref{thm:MainResultsScheme} cannot be in general inherited since the converge of each numerical solution can be true only for different subsequences depending on the initial condition. However, if uniqueness of viscosity solutions holds, the estimate
\[
	\|u(\cdot, t)-v(\cdot, t)\|_{L^\infty(\R^d)}\leq \|u_0-v_0\|_{L^\infty(\R^d)},
\]
is also true under the assumptions of Theorem \ref{thm:MainResultsViscosity}.

Having proved the existence of a viscosity solution, we can additionally obtain estimates on orders of convergence for the numerical scheme. 

\begin{thm}\label{thm:MainResultsOrders}
Let the assumptions of Theorem \ref{thm:MainResultsScheme} hold along with \eqref{As:constrong} and \eqref{As:Le} for some $r\in[0,1)$, and $u_0\in X^a_b(\R^d)$ with $a\in(r,1]$. Let $e^n(x) = u(x,t_n) - U^n(x)$ be the error of the numerical scheme, where $u$ is the function given in Corollary \ref{thm:MainResultsConvergence} with $U^0 = u_0$. Then,
\[
	\|e^n\|_{L^\infty(\mathbb{R}^d)} \leq C\left(t_n^{2\alpha-1} \tau^{1-\alpha}
	+t_n^{\alpha}h^{a-r}\right),
\]
for some constant $C=C(\alpha,r, T, D, u_0)>0$.
\end{thm}

\begin{rem}
Note that the above error in time $\|e^n\|_{L^\infty(\mathbb{R}^d)}$ gives different behaviors depending on the range of $\alpha$ and for short and long times. More precisely, for $\alpha \in[1/2,1)$, the error in time is $O(\tau^{1-\alpha})$ for all times. However, when $\alpha \in (0,1/2)$, let us fix $t^*>0$ and then the error in time is bounded by
\[
C\left\{
\begin{split}
 \frac{\tau^\alpha}{n^{1-2\alpha}}, \quad &\textup{if}\quad  t<t^*,\\
\frac{\tau^{1-\alpha}}{t_n^{1-2\alpha}}, \quad &\textup{if} \quad t\geq t^*.
\end{split}
\right.
\]
This means that the scheme is of order $O(\tau^\alpha)$ for short times. However, for large times the scheme has accuracy $O(\tau^{1-\alpha})$ for large times, which in this range is better.
\end{rem}

Further below, in Section \ref{sec:Orders}, we present more detailed results concerning orders of convergence. In what follows, we prove all the above statements and later verify the discrete scheme by several numerical experiments. 

\section{Properties of the numerical scheme}\label{sec:PropNumSch}

We first formulate the result of the existence and uniqueness of numerical solutions.
\begin{thm}\label{thm:exisun} Assume \eqref{As:wepos}.
Let $U^0 \in \mathcal{B}(\R^d)$, $D\in\mathcal{B}(Q_T)$ nonnegative and $h,\tau>0$. There exists a unique solution $U^n\in \mathcal{B}(\R^d)$ of \eqref{eq:NumSch} for all $n\geq0$.
\end{thm}
\begin{proof} 

Since we are dealing with a uniform grid, it is enough to show that, given $\{\phi^j\}^{n-1}_{j=0}\subset \ell^\infty(h\Z^d)$, there exists $\phi^n\in \ell^\infty(h\Z^d)$ such that
\[
\frac{\tau^{-\alpha}}{\Gamma(2-\alpha)}\left( \phi^n_\gamma - b_{n-1}\phi^0_\gamma -\sum_{k=1}^{n-1}(b_{n-k-1}-b_{n-k})\phi^k_\gamma \right)= D^n_\gamma \sum_{\beta\not=0} \left(\phi^n_{\gamma+\beta}+\phi^n_{\gamma-\beta}-2\phi^n_\gamma \right)\omega_\beta(h).
\]
Rewriting the above identity, we get,
\[
\phi_\gamma^n = \frac{ D_\gamma^n \Gamma(2-\alpha)\tau^\alpha \sum_{\beta\not=0} \left(\phi^n_{\gamma+\beta}+\phi^n_{\gamma-\beta} \right)\omega_\beta(h)  +  b_{n-1}\phi^0_\gamma +\sum_{k=1}^{n-1}(b_{n-k-1}-b_{n-k})\phi^k_\gamma}{1+  2D_\gamma^n\Gamma(2-\alpha)\tau^\alpha\sum_{\beta\not=0} w_\beta(h)}.
\]
We then need to find a solution $\psi \in \ell^\infty(h\Z^d)$ of the following fixed point problem:
\[
\psi=T[\psi] \quad \textup{with} \quad T[\psi]_\gamma=\frac{   D_\gamma^n \Gamma(2-\alpha)\tau^\alpha \sum_{\beta\not=0} \left(\psi_{\gamma+\beta}+\psi_{\gamma-\beta} \right)\omega_\beta(h)}{1+  2D_\gamma^n\Gamma(2-\alpha)\tau^\alpha\sum_{\beta\not=0} w_\beta(h)} + f_\gamma,
\]
with $f \in \ell^\infty(h\Z^d)$ given by
\[
f_\gamma= \frac{ b_{n-1}\phi^0_\gamma +\sum_{k=1}^{n-1}(b_{n-k-1}-b_{n-k})\phi^k_\gamma}{1+  2D_\gamma^n\Gamma(2-\alpha)\tau^\alpha\sum_{\beta\not=0} w_\beta(h)}.
\]
Clearly,
\[
\begin{split}
	\|T[\psi]\|_{\ell^\infty(h\Z^d)}&\leq \sup_{\gamma\in \Z}\left\{\frac{2D_\gamma^n\Gamma(2-\alpha)\tau^\alpha\sum_{\beta\not=0} w_\beta(h)}{1+  2D_\gamma^n\Gamma(2-\alpha)\tau^\alpha\sum_{\beta\not=0} w_\beta(h)} \right\}\|\psi\|_{\ell^\infty(h\Z^d)} + \|f\|_{\ell^\infty(h\Z^d)}\\
	&\leq \|\psi\|_{\ell^\infty(h\Z^d)} + \|f\|_{\ell^\infty(h\Z^d)}.
\end{split}
\]
Finally note that, since $x\mapsto \frac{x}{1+x}$ is nondecreasing, we get strict contractivity of the map as follows
\[
\begin{split}
	\|T[\psi]-T[\varphi]\|_{\ell^\infty(h\Z^d)}&\leq \sup_{\gamma\in \Z}\left\{ \frac{2D_\gamma^n\Gamma(2-\alpha)\tau^\alpha\sum_{\beta\not=0} w_\beta(h)}{1+  2D_\gamma^n\Gamma(2-\alpha)\tau^\alpha\sum_{\beta\not=0} w_\beta(h)}\right\} \|\psi-\varphi\|_{\ell^\infty(h\Z^d)}\\
	&\leq  \frac{2\|D^n\|_{\ell^\infty(h\Z^d)}\Gamma(2-\alpha)\tau^\alpha\sum_{\beta\not=0} w_\beta(h)}{1+  2\|D^n\|_{\ell^\infty(h\Z^d)}\Gamma(2-\alpha)\tau^\alpha\sum_{\beta\not=0} w_\beta(h)} \|\psi-\varphi\|_{\ell^\infty(h\Z^d)}\\
	&=L  \|\psi-\varphi\|_{\ell^\infty(h\Z^d)},
\end{split}
\]
with $L<1$. Thus, by the Banach fixed point theorem, there exists a unique solution to the problem.
\end{proof}

\subsection{Properties of the numerical scheme: Equicontinuity in space}
First we show stability of $L_h$ for our scheme. More precisely:
\begin{lem}\label{lem:decLh}
Assume \eqref{As:wepos}, \eqref{As:webdd} for some $r\in[0,2]$, and $h,\tau>0$. Let $U^0\in BUC(\R^d)$, $D\in C_b([0,T]: X^r_b(\R^d))$ nonnegative, and $U^n$ satisfy \eqref{eq:NumSch}. Then, there exists $\tau_0$ such that, for all $\tau\leq \tau_0$, we have
\begin{equation}
	\|L_hU^n\|_{L^\infty(\R^d)} \leq C \|L_hU^0\|_{L^\infty(\R^d)}, \quad \textup{for all} \quad t_n\in [0,T],
\end{equation}
where $C=C(\alpha,r, T, \|D\|_{L^\infty(Q_T)},\|D\|_{X^r(Q_T)})>0$.
\end{lem}
\begin{proof}
For any $\beta\in \Z^d$, we combine equation \eqref{eq:NumSch} posed in $x$, $x+z_\beta$ and $x-z_\beta$ and multiply by $\omega_\beta(h)$ to get
\[
\begin{split}
	\partial_\tau^\alpha ((U^n(x+z_\beta)&+U^n(x-z_\beta)-2U^n(x))\omega_\beta(h) )\\
	&=(D^n(x+z_\beta)L_h U^n(x+z_\beta)+D^n(x-z_\beta)L_h U^n(x-z_\beta)- 2D^n(x)L_h U^n(x)) \omega_\beta(h).
\end{split}
\]
Reordering the above leads to
\[
\begin{split}
	\partial_\tau^\alpha &((U^n(x+z_\beta)+U^n(x-z_\beta)-2U^n(x))\omega_\beta(h) )\\
	&+ D^n(x+z_\beta) (L_h U^n(x)-L_h U^n(x+z_\beta))\omega_\beta(h)+ D^n(x-z_\beta) (L_h U^n(x)-L_h U^n(x-z_\beta))\omega_\beta(h)\\
	=&(D^n(x+z_\beta)+D^n(x-z_\beta)- 2D^n(x)) \omega_\beta(h) L_h U^n(x).
\end{split}
\]
Summing in $\beta$ and writing explicitly $\partial_\tau^\alpha$, we get
\[
\begin{split}
	&\frac{\tau^{-\alpha}}{\Gamma(2-\alpha)} L_hU^n(x) \\
	&+ \sum_{\beta\not=0} \left(D^n(x+z_\beta) (L_h U^n(x)-L_h U^n(x+z_\beta))\omega_\beta(h)+ D^n(x-z_\beta) (L_h U^n(x)-L_h U^n(x-z_\beta))\omega_\beta (h)\right)\\
	=&L_hD^n(x) L_h U^n(x) + \frac{\tau^{-\alpha}}{\Gamma(2-\alpha)} \left(b_{n-1}L_hU^0(x) + \sum_{k=1}^{n-1} (b_{n-k-1} - b_{n-k}) L_hU^k(x)\right).
\end{split}
\]
Now, by Theorem \ref{thm:exisun} and \eqref{As:wepos}, we can consider a sequence $\{x_\veps\}_{\veps>0}\subset \mathbb{R}^d$ such that
\[
L_hU^n(x_\veps) \to \esssup_{x\in \mathbb{R}^d} \{L_hU^n(x)\}\quad \textup{as} \quad \veps\to0^+.
\]
Note that, in this way,
\[
\lim_{\veps\to0^+}(L_h U^n(x_\veps)-L_h U^n(x_\veps-z_\beta))\geq0, \quad \textup{for all} \quad \beta \in \Z^d.
\]
Using the above information, we can conclude that
\[
\begin{split}
	\frac{\tau^{-\alpha}}{\Gamma(2-\alpha)} \esssup_{x\in \mathbb{R}^d} \{L_hU^n(x)\} \leq&  \frac{\tau^{-\alpha}}{\Gamma(2-\alpha)} \left( b_{n-1}\|L_hU^0\|_{L^\infty(\R^d)} + \sum_{k=1}^{n-1} (b_{n-k-1} - b_{n-k}) \|L_hU^k\|_{L^\infty(\R^d)}\right)\\
	&+\|L_hD^n\|_{L^\infty(\R^d)}\|L_h U^n\|_{L^\infty(\R^d)}.
\end{split}
\]
In a similar way, we can conclude
\[
\begin{split}
	\frac{\tau^{-\alpha}}{\Gamma(2-\alpha)} \essinf_{x\in \mathbb{R}^d} \{L_hU^n(x)\} \geq&  -\frac{\tau^{-\alpha}}{\Gamma(2-\alpha)} \left( b_{n-1}\|L_hU^0\|_{L^\infty(\R^d)}+ \sum_{k=1}^{n-1} (b_{n-k-1} - b_{n-k}) \|L_hU^k\|_{L^\infty(\R^d)}\right)\\
	&-\|L_hD^n\|_{L^\infty(\R^d)}\|L_h U^n\|_{L^\infty(\R^d)},
\end{split}
\]
which implies
\[
\begin{split}
	\frac{\tau^{-\alpha}}{\Gamma(2-\alpha)} \|L_hU^n\|_{L^\infty(\R^d)} \leq&  \frac{\tau^{-\alpha}}{\Gamma(2-\alpha)} \left( b_{n-1}\|L_hU^0\|_{L^\infty(\R^d)} + \sum_{k=1}^{n-1} (b_{n-k-1} - b_{n-k}) \|L_hU^k\|_{L^\infty(\R^d)}\right)\\
	&+\|L_hD^n\|_{L^\infty(\R^d)}\|L_h U^n\|_{L^\infty(\R^d)},
\end{split}
\]
that is,
\begin{equation}\label{eq:CaputoInequality}
	\partial_\tau^\alpha(\|L_hU^n\|_{L^\infty(\R^d)})\leq \|L_hD^n\|_{L^\infty(\R^d)}\|L_h U^n\|_{L^\infty(\R^d)}.
\end{equation}
By assumption \eqref{As:webdd}, the regularity of $D$ and Remark \ref{rem:boundwei}, we get
\[
\partial_\tau^\alpha(\|L_hU^n\|_{L^\infty(\R^d)})\leq C_r (\| D\|_{C_b([0,T]: X^r(\R^d))} +4 \| D\|_{L^\infty(Q_T)}  )\|L_h U^n\|_{L^\infty(\R^d)}.
\]
By Gr\"onwall inequality \eqref{eq:GrongCap1} with $\lambda_0=C_r (\| D\|_{C_b([0,T]: X^r(\R^d))} +4 \| D\|_{L^\infty(Q_T)}  )$, $\lambda_1=0$, $F^n=0$ and $\tau_0 =1/(2\lambda_0 \Gamma(2-\alpha))^{\frac{1}{\alpha}}$, we get
\[
\|L_hU^n\|_{L^\infty(\R^d)} \leq 2 E_\alpha({C t_n^\alpha })\|L_hU^0\|_{L^\infty(\R^d)}
\]
where the constant $C$ is the one specified in \eqref{eq:GrongCap1}. The result follows since the Mittag-Leffler function is nondecreasing.

\end{proof}

We give now our first partial result on equicontinuity in space. We note that the modulus of continuity will not be uniform in the discretization parameters unless we have a uniform bound of $\|L_h U^0\|_{L^\infty(\R^d)} $. We specify this later in a corollary. 

\begin{thm}\label{thm:equicont1}
Assume \eqref{As:wepos}, \eqref{As:webdd} for some $r\in[0,2]$ and $h,\tau>0$. Let $U^0\in BUC(\R^d)$, $D\in C_b([0,T]: X^r_b(\R^d))$ nonnegative, and $U^n$ satisfy \eqref{eq:NumSch}. Then, 
\[
\|U^n(\cdot+y)-U^n\|_{L^\infty(\R^d)}\leq \|U^0(\cdot+y)-U^0\|_{L^\infty(\R^d)}+  C \|L_h U^0\|_{L^\infty(\R^d)} t_n^{\alpha} \sup_{t\in[0,T]} \|D(\cdot+y,t)-D(\cdot,t)\|_{L^\infty(\R^d)} ,
\]
where $C=C(\alpha,r, T, \|D\|_{L^\infty(Q_T)},\|D\|_{X^r(Q_T)})>0$.
\end{thm}

\begin{proof}
For every $y\in \mathbb{R}^d$, define $\delta_y U^n(x)=U^n(x+y)-U^n(x) $. We combine equation \eqref{eq:NumSch} posed in $x$, $x+y$ to get
\[
\partial_\tau^\alpha(\delta_y U^n(x))- D^n(x+y) L_h U^n(x+y)+D^n(x) L_h U^n(x)=0,
\]
that is,
\[
\partial_\tau^\alpha(\delta_y U^n(x))+ D^n(x+y) (L_h U^n(x)-L_h U^n(x+y))-\delta_y D^n(x) L_h U^n(x)=0,
\]
where we have used that $\partial_\tau^\alpha U^n(x+y) -\partial_\tau^\alpha U^n(x)=\partial_\tau^\alpha(\delta_y U^n(x))$. We also note that
\[
\begin{split}
	L_h U^n(x)-L_h U^n(x+y)&=\sum_{\beta\not=0} \left(U^n(x+y_\beta)-U^n(x)\right)\omega_\beta(h)-\sum_{\beta\not=0}\left(U^n(x+y+y_\beta)-U^n(x+y)\right)\omega_\beta(h)\\
	&=\sum_{\beta\not=0}\left( \delta_y U^n(x) -\delta_y U^n(x+y_\beta) \right)\omega_\beta(h).
\end{split}
\]
Therefore,
\[ 
\partial_\tau^\alpha(\delta_y U^n(x))+ D^n(x+y) \sum_{\beta\not=0} \left( \delta_y U^n(x) -\delta_y U^n(x+y_\beta) \right)\omega_\beta(h)-\delta_y D^n(x) L_h U^n(x)=0.
\]
Arguing as in the proof of Lemma \ref{lem:decLh} for the $\esssup$ and $\essinf$ of $\delta_y U^n$, and using Lemma \ref{lem:decLh}, we get
\[
\begin{split}
	\partial_\tau^\alpha(\|\delta_y U^n\|_{L^\infty(\R^d)})&\leq  \|\delta_y D^n\|_{L^\infty(\R^d)} \|L_h U^n\|_{L^\infty(\R^d)}\leq C \|D(\cdot+y,\cdot)-D\|_{L^\infty(Q_T)} \|L_h U^0\|_{L^\infty(\R^d)}.
\end{split}
\]
Finally, using Gr\"onwall inequality \eqref{eq:Gronw2}, we get
\[
\|\delta_y U^n\|_{L^\infty(\R^d)} \leq \|\delta_y U^0\|_{L^\infty(\R^d)}+ \frac{\Gamma(2-\alpha)}{\alpha}C \|D(\cdot+y,\cdot)-D\|_{L^\infty(Q_T)} \|L_h U^0\|_{L^\infty(\R^d)}t_n^{\alpha},
\]
which finishes the proof. 
\end{proof}

We present now the proof of equicontinuity in space when the initial data $U^0\in X^r_b(\R^d)$, the natural space for which $L_h$ is bounded uniformly in $h$.
\begin{cor}
Assume \eqref{As:wepos}, \eqref{As:webdd} for some $r\in[0,2]$ and $h,\tau>0$. Let $U^0\in  X^r_b(\R^d)$, $D\in C_b([0,T]: X^r_b(\R^d))$ nonnegative, and $U^n$ satisfy \eqref{eq:NumSch}. Then, there exists $\tau_0$ such that for all $\tau\leq \tau_0$, we have
\[
\|U^n(\cdot+y)-U^n\|_{L^\infty(\R^d)}\leq \|U^0(\cdot+y)-U^0\|_{L^\infty(\R^d)}+  C t_n^{\alpha} \sup_{t\in[0,T]} \|D(\cdot+y,t)-D(\cdot,t)\|_{L^\infty(\R^d)} ,
\]
where $C=C(\alpha,r, T, \|D\|_{L^\infty(Q_T)},\|D\|_{X^r(Q_T)}, \|U^0\|_{L^\infty(\R^d)},\|U^0\|_{X^r(\R^d)})>0$.
\end{cor}
\begin{proof}
The result follows from Theorem \ref{thm:equicont1} and the uniform bound of $L_h$ given in Remark \ref{rem:boundwei}.
\end{proof}

We can also prove such results for less regular initial data. First, we need the following $L^\infty$ contraction result. 

\begin{lem}\label{lem:contraction}
Assume \eqref{As:wepos} and $h,\tau>0$. Let  $D\in \mathcal{B}(Q_T)$ nonnegative and $U^n,V^n$ satisfy \eqref{eq:NumSch} with initial data $U^0,V^0\in \mathcal{B}(\R^d)$, respectively. Then,
\[
\|U^n-V^n\|_{L^\infty(\R^d)}\leq \|U^0-V^0\|_{L^\infty(\R^d)}.
\]
\end{lem}
\begin{proof}
Let $W^n=U^n-V^n$. By linearity we have
\[
\partial^\alpha_\tau W^n(x)= D^n(x) L_h W^n(x).
\]
Arguing as before, we can get
\[
\partial_{\tau}^\alpha(\|W^n\|_{L^\infty(\R^d)})\leq 0,
\]
which, by Gr\"onwall inequality \eqref{eq:Gronw2}, implies the result.
\end{proof}

We are now ready to prove our equicontinuity result for general initial data $U^0\in X^a_b(\R^d)$ for any $a\leq r$. To prove such result, we introduce the notation of modulus of continuity, i.e.,

\[
	\Lambda_{{U^0}}(\eta)=\sup_{|z|\leq \eta}\|U^0(\cdot+z)-U^0\|_{L^\infty(\R^d)} \quad \textup{and} \quad \Lambda_{D}(\eta) = \sup_{|z|\leq\eta} \sup_{t\in[0,T]} \|D(\cdot+z,t)-D(\cdot,t)\|_{L^\infty(\R^d)}.
\]

\begin{cor}\label{cor:regspace}
Assume \eqref{As:wepos}, \eqref{As:webdd} for some $r\in[0,2]$ and $h,\tau>0$. Let $U^0\in X^a_b(\R^d)$ for some $a\in(0,r]$, $D\in C_b([0,T]: X^r_b(\R^d))$ nonnegative, and $U^n$ satisfy \eqref{eq:NumSch}. Then, there exists $\tau_0$ such that, for all $\tau\leq \tau_0$, we have
\[
\|U^n(\cdot+y)-U^n\|_{L^\infty(\R^d)} \leq C \max\{1, t_{n}^\alpha\} \left\{ \begin{split}
	 \max\{|y|^a, |y|^r\}, \quad  & \textup{if} \quad r\in (0,1],\\
	 \max\{|y|^\frac{a}{r},|y|\}, \quad  & \textup{if} \quad r\in (1,2], \quad
\end{split}\right.
\]
for some constant $C=C(\alpha,r, T, \|D\|_{L^\infty(Q_T)},\|D\|_{X^r(Q_T)}, \|U^0\|_{L^\infty(\R^d)},\|U^0\|_{X^a(\R^d)})>0$.
\end{cor}
\begin{proof}
The result will follow again from Theorem \ref{thm:equicont1}. Consider a standard mollifier $\rho_\delta$ for $\delta >0$ (see , e.g., Appendix A in \cite{dTLi22} for a suitable choice of the mollifying function). Let $U^0_\delta=U^0*\rho_\delta \in C^\infty_b(\R^d)$ and $(U_\delta)^n$ be the corresponding solution of \eqref{eq:NumSch}.
By Lemma \ref{lem:contraction},  the regularity of ${U^0}$, and the properties of the mollifiers, there exists a constant $L_{{U^0}}>0$ such that
\[
\|U^n-(U_\delta)^n\|_{L^\infty(\R^d)}\leq \|U^0-U^0_\delta\|_{L^\infty(\R^d)} \leq L_{{U^0}}\delta^a.
\]
On the other hand, by the regularity of $U^0$, we have (see, e.g., Lemma A.1 in \cite{dTLi22}) 
\[
\|U^0_\delta(\cdot+y)-U^0_\delta\|_{L^\infty(\R^d)}\leq \|U^0_\delta\|_{C^1(\R^d)}|y| \leq K_1 \Lambda_{{U^0}}(\delta)|y|\delta^{-1},
\]
where $K_1$ is a constant depending only on the dimension and the choice of the mollifier. Thus, by Theorem \ref{thm:equicont1}, and the regularity of $D$, there exists a constant $L_{D}>0$ such that
\[
\begin{split}
	\|(U_\delta)^n(\cdot+y)-(U_\delta)^n\|_{L^\infty(\R^d)}\leq& \|U^0_\delta(\cdot+y)-U^0_\delta\|_{L^\infty(\R^d)}\\
	&+  C \|L_h U^0_\delta\|_{L^\infty(\R^d)} t_n^{\alpha} \sup_{t\in[0,T]} \|D(\cdot+y,t)-D(\cdot,t)\|_{L^\infty(\R^d)}\\
	\leq&K_1 \Lambda_{{U^0}}(\delta) |y|\delta^{-1} + \tilde{C} (\|{U^0_\delta}\|_{X^r(\R^d)}+ \|U_\delta^0\|_{L^\infty(\R^d)})t_n^\alpha \Lambda_{D}(|y|)\\
	\leq& K_1  \Lambda_{{U^0}}(\delta) |y|\delta^{-1} + \tilde{C} (L_{{U^0}}\delta^{a-r}+\|U^0\|_{L^\infty(\R^d)} )t_n^\alpha L_D |y|^{\min\{1,r\}}.
\end{split}
\]
Using the above estimates, 
\[
\begin{split}
	\|U^n(\cdot+y)-U^n\|_{L^\infty(\R^d)}\leq& \|U^n(\cdot+y)-(U_\delta)^n(\cdot+y)\|_{L^\infty(\R^d)}+\|U^n-(U_\delta)^n\|_{L^\infty(\R^d)}\\
	&+\|(U_\delta)^n(\cdot+y)-(U_\delta)^n\|_{L^\infty(\R^d)}\\
	\leq& 2L_{{U^0}}\delta^a +  K_1  \Lambda_{{U^0}}(\delta) |y|\delta^{-1} + \tilde{C} (L_{{U^0}}\delta^{a-r}+ L_{{U^0}})t_n^\alpha L_D |y|^{\min\{1,r\}}.\end{split}
\]

If $r\in (0,1]$ (and, thus, $a\leq1$), we have $\Lambda_{{U^0}}(\delta)\leq L_{{U^0}} \delta^a$.   In what follows we will use $K$ to denote a positive constant, and it might change from inequality to inequality to keep the notation as simple a possible.  Then, the above estimate reads
\[
\|U^n(\cdot+y)-U^n\|_{L^\infty(\R^d)} \leq  K  \left(\delta^a+ |y| \delta^{a-1} +  t_{n}^{\alpha}(\delta^{a-r}+1) |y|^{r}\right).
\]
It is enough to take $\delta=|y|$ to get
\[
\|U^n(\cdot+y)-U^n\|_{L^\infty(\R^d)} \leq  K  ( |y|^a +  t_{n}^{\alpha}(|y|^a+|y|^r) ).
\]

If $r\in(1,2]$ then we need to consider two cases. If $a\leq1$, we have that
\[
\|U^n(\cdot+y)-U^n\|_{L^\infty(\R^d)} \leq   K  \left(\delta^a+ |y| \delta^{a-1} +  t_{n}^{\alpha}(\delta^{a-r}+1) |y|\right),
\]
and take $\delta=|y|^{\frac{1}{r}}$ to get
\[
\|U^n(\cdot+y)-U^n\|_{L^\infty(\R^d)} \leq   K  \left(|y|^{\frac{a}{r}}+ |y|^{\frac{a}{r}+1-\frac{1}{r}} +  t_{n}^{\alpha}(|y|^{\frac{a}{r}}+|y|)\right).
\]
Note that, if $|y|\leq1$, $|y|^{\frac{a}{r}}>|y|^{\frac{a}{r}+1-\frac{1}{r}}$, while if $|y|>1$, $|y|>|y|^{\frac{a}{r}+1-\frac{1}{r}} $ and the result follows.
Finally, if $a>1$, we have $\Lambda_{{U^0}}(\delta)\leq L_{{U^0}} \delta$. Then, taking again $\delta=|y|^{\frac{1}{r}}$
\[
\begin{split}
	\|U^n(\cdot+y)-U^n\|_{L^\infty(\R^d)} &\leq   K  \left( \delta^a + |y| +  t_{n}^{\alpha}(\delta^{a-r}+1) |y|\right)\\
	&\leq    K \left( |y|^{\frac{a}{r}} + |y| +  t_{n}^{\alpha}(|y|^{\frac{a}{r}}+|y|)\right).
\end{split}
\]
\end{proof}

\subsection{Properties of the numerical scheme: Equicontinuity in time}
In this section, we characterize the regularity in time of the solution $U^n$ of the numerical scheme \eqref{eq:NumSch}. We start by introducing the discrete (L1) Riemann-Liouville fractional derivative
\begin{equation}\label{eq:discRLder}
^{RL} \partial_\tau^\alpha f^n:= \frac{\tau^{-\alpha}}{\Gamma(2-\alpha)}\left(f^n - \sum_{k=0}^{n-1} (b_{n-k-1}-b_{n-k})f^k \right),
\end{equation}
where $b_i$ are as in \eqref{def:disccap}. The above formulation arises in the study of the difference of the discrete Caputo derivative.
\begin{lem}
Let us define $\delta_1 f^n=f^{n+1}-f^n$ for any sequence of numbers $(f^n)_n$. Then
\[
\delta_1 \partial_\tau^\alpha f^n= {^{RL}\partial_\tau^\alpha} \delta_1 f^n.
\]
\end{lem}
\begin{proof}
By a straightforward computation we have
\[
\begin{split}
	\delta_1 \partial_\tau^\alpha f^n&= \partial_\tau^\alpha f^{n+1}-\partial_\tau^\alpha f^{n}\\
	&=\frac{\tau^{-\alpha}}{\Gamma(2-\alpha)}\left(\delta_1 f^n - f^0(b_n-b_{n-1}) -  \sum_{k=1}^n (b_{n-k}-b_{n+1-k})f^k + \sum_{k=1}^{n-1} (b_{n-k-1}-b_{n-k})f^k \right) \\
	&=\frac{\tau^{-\alpha}}{\Gamma(2-\alpha)}\left(\delta_1 f^n - f^0(b_n-b_{n-1}) -  \sum_{k=0}^{n-1} (b_{n-k-1}-b_{n-k})f^{k+1} + \sum_{k=1}^{n-1} (b_{n-k-1}-b_{n-k})f^k \right) \\
	&= \frac{\tau^{-\alpha}}{\Gamma(2-\alpha)}\left(\delta_1 f^n  - \sum_{k=0}^{n-1}(b_{n-k-1}-b_{n-k})\delta_1f^{k}\right)\\
	&={^{RL}\partial_\tau^\alpha} \delta_1 f^n.
\end{split}
\]
The proof is complete. 
\end{proof}
Consider now the main scheme \eqref{eq:NumSch}. Writing it once at time $t_n$ and once at $t_{n+1}$ and subtracting, we obtain, for $n\geq1$, the following equation
\[
\begin{split}
{^{RL}\partial_\tau^\alpha} \delta_1 U^n(x)
&= \partial_\tau^\alpha U^{n+1}-\partial_\tau^\alpha U^{n}\\
&= D^{n+1}(x) L_hU^{n+1}(x)-D^{n}(x) L_hU^{n}(x)\\
&=\delta_1 D^n(x) L_h U^{n+1}(x) + D^n(x)  L_h \delta_1U^n(x),
\end{split}
\]
that is,
\begin{equation}\label{eq:RLtransf}
^{RL}\partial_\tau^\alpha \delta_1 U^n(x) = \delta_1 D^n(x) L_h U^{n+1}(x) + D^n(x)  L_h \delta_1U^n(x),\quad \textup{for all} \quad n\geq1.
\end{equation}
Moreover, directly from \eqref{eq:NumSch} for $n=1$, we have 
\begin{equation}\label{eq:inicondRL}
\delta_1 U^0(x) = \Gamma(2-\alpha)\tau^\alpha D^{1}(x) L_h U^1(x).
\end{equation}
Note that the above discrete relation between the Riemann-Liouville and Caputo derivatives has its continuous counterpart that follows directly from the definition, that is
\[
\partial_t \partial_t^\alpha f(t)= {^{RL}\partial_t^\alpha} f'(t).
\]
where $^{RL}\partial_t^\alpha$ the Riemann-Liouville derivative. Now we can state the main results describing the regularity in time of the solution to the numerical scheme \eqref{eq:NumSch}.
\begin{thm}\label{thm:RegularityTime}
Assume \eqref{As:wepos}, \eqref{As:webdd} for some $r\in[0,2]$ and $h,\tau>0$. Let $U^0\in BUC(\R^d)$, $D\in C_b([0,T]: X^r_b(\R^d))\cap C_b(\R^d: Lip([0,T]))$ nonnegative and $U^n$ satisfy \eqref{eq:NumSch}. Then, we have
\begin{equation}\label{eqn:RegularityTime}
	\tau^{-1}\|U^n-U^{n-1}\|_{L^\infty(\mathbb{R}^d)} \leq C\|L_h U^0\|_{L^\infty(\R^d)}\left(1+t_n^{\alpha-1}\right),
\end{equation}
for all $n\geq1$, where $C=C(\alpha,r, T, \|D\|_{L^\infty(Q_T)},\|D\|_{C_b([0,T]: X^r_b(\R^d))}, \|D\|_{C_b(\R^d: Lip([0,T]))})>0$.
\end{thm}
\begin{proof}
First, arguing as in the proof of Lemma \ref{lem:decLh} for the $\esssup$ and $\essinf$ of $\delta_1 U^n$, and using Lemma \ref{lem:decLh} in \eqref{eq:RLtransf} we get
\begin{equation}
	\label{eq:RLInequality}
	^{RL}\partial_\tau^\alpha \|\delta_1 U^{n-1}\|_{L^\infty(\mathbb{R}^d)} \leq C \|\delta_1 D^{n-1}\|_{L^\infty(Q_T)} \|L_h U^{0}\|_{L^\infty(\mathbb{R}^d)} \leq C \tau \|L_h U^{0}\|_{L^\infty(\mathbb{R}^d)},
\end{equation}
for $n\geq1$, and directly from \eqref{eq:inicondRL} and Lemma \ref{lem:decLh},
\begin{equation}\label{eq:caso0}
\|\delta_1 U^0\|_{L^\infty(\mathbb{R}^d)} \leq C \tau^\alpha\|L_h U^{0}\|_{L^\infty(\mathbb{R}^d)},
\end{equation}
where the constant $C=C(\alpha,r, T, \|D\|_{L^\infty(Q_T)},\|D\|_{X^r(Q_T)})>0$.  Note that \eqref{eqn:RegularityTime} for $n=1$ follows directly from \eqref{eq:caso0} since $t_1=\tau$.  On the other hand, for $n\geq2$,  an application of the Riemann-Liouville version of the Gr\"onwall inequality \eqref{eq:RLCor} yields
\[
\begin{split}
\|\delta_1 U^{n-1}\|_{L^\infty(\mathbb{R}^d)} &\leq C \|L_h U^{0}\|_{L^\infty(\mathbb{R}^d)} \left(\tau^\alpha\tau^{1-\alpha} t^{\alpha-1}_{n-1} + \frac{\Gamma(2-\alpha)}{\alpha}\tau t^\alpha_{n-1}\right)\\
& \leq C \|L_h U^{0}\|_{L^\infty(\mathbb{R}^d)} \tau \left(t^{\alpha-1}_{n-1} + \frac{\Gamma(2-\alpha)}{\alpha} T^\alpha\right).
\end{split}
\]
Finally note that, for $n\geq2$, we have $t_{n-1}^{\alpha-1}= t_{n}^{\alpha-1}\left(\frac{n}{n-1}\right)^{1-\alpha}\leq 2 t_{n}^{\alpha-1}$, 
from which the desired result follows.  
\end{proof}

From this, it is now possible to obtain a modulus of continuity in time. 
\begin{cor}\label{cor:EquicontinuityTime}
Assume \eqref{As:wepos}, \eqref{As:webdd} for some $r\in(0,2]$ and $h,\tau>0$. Let $U^0\in X^a_b(\R^d)$ for some $a\in(0,r]$, $D\in C_b([0,T]: X^r_b(\R^d)) \cap C_b(\R^d: Lip([0,T]))$ nonnegative and $U^n$ satisfy \eqref{eq:NumSch}. Then, for any $0\leq m<n$, we have
\begin{equation}\label{equi:est1}
	\|U^n - U^m\|_{L^\infty(\mathbb{R}^d)} \leq C \max\{(t_n-t_m)^{\alpha\frac{a}{r}}, (t_n-t_m)^{\frac{a}{r}}\},
\end{equation}	
and 
\begin{equation}\label{equi:est2}
	\|U^n - U^m\|_{L^\infty(\mathbb{R}^d)} \leq  C ((t_n^{\alpha-1}+1)(t_n-t_m) )^\frac{a}{r},
\end{equation}
where $C=C(\alpha,r, T, \|D\|_{L^\infty(Q_T)},\|D\|_{C_b([0,T]: X^r_b(\R^d))}, \|D\|_{C_b(\R^d: Lip([0,T]))}, \|U^0\|_{L^\infty(\R^d)},\|U^0\|_{X^a(\R^d)})>0$.
\end{cor}
\begin{proof}
Without loss of generality, we can consider the difference $\|U^{m+k} - U^m\|_{L^\infty(\mathbb{R}^d)}$ for any $m,k\in\mathbb{N}$. We can then iterate one-step differences and use Theorem \ref{thm:RegularityTime} to obtain the bound
\begin{equation}
	\|U^{m+k} - U^m\|_{L^\infty(\mathbb{R}^d)} \leq \sum_{j=1}^k \|U^{m+j} - U^{m+j-1}\|_{L^\infty(\mathbb{R}^d)} \leq C{\|L_h U^{0}\|_{L^\infty(\mathbb{R}^d)}} \tau \sum_{j=1}^k (t_{m+j}^{\alpha-1} + 1).
\end{equation}
Now, one sum above is trivial to evaluate yielding $\tau \sum_{j=1}^k 1 = t_k$, while the other can be bounded by an integral
\begin{equation}
	\tau \sum_{j=1}^k t_{m+j}^{\alpha-1} \leq \int_{t_m}^{t_{m+k}} t^{\alpha-1} dt = \frac{1}{\alpha} \left(t_{m+k}^\alpha - t_{m}^\alpha\right),
\end{equation}
since the function $t\mapsto t^{\alpha-1}$ is decreasing. Therefore,
\begin{equation}
	\|U^{m+k} - U^m\|_{L^\infty(\mathbb{R}^d)} \leq C \|L_h U^{0}\|_{L^\infty(\mathbb{R}^d)} \max\{t_{m+k}^\alpha-t_{k}^\alpha,t_{m+k}-t_{m}\}.
\end{equation}
In what follows we will use $K$ to denote a positive constant, and it might change from inequality to inequality to keep the notation as simple a possible. Proceeding by mollification, as in the proof of Corollary \ref{cor:regspace}, and using the above estimate for $U_\delta$, we get
\[
\begin{split}
	\|U^{m+k} - U^m\|_{L^\infty(\mathbb{R}^d)}\leq& \|U^{m+k} - (U_\delta)^{m+k}\|_{L^\infty(\mathbb{R}^d)}+\|U^{m} - (U_\delta)^{m}\|_{L^\infty(\mathbb{R}^d)}+\|(U_\delta)^{m+k} - (U_\delta)^{m}\|_{L^\infty(\mathbb{R}^d)}\\
	\leq&  2\|U^{0} - U_\delta^{0}\|_{L^\infty(\mathbb{R}^d)}+C { \|L_h U_\delta^{0}\|_{L^\infty(\mathbb{R}^d)}}\max\{t_{m+k}^\alpha-t_{m}^\alpha,t_{m+k}-t_{m}\}\\
	\leq & 2L_{{U^0}} \delta^a + \tilde{C}(L_{{U^0}}\delta^{a-r}+ \|U_\delta^0\|_{L^\infty(\R^d)}) \max\{t_{m+k}^\alpha-t_{m}^\alpha,t_{m+k}-t_{m}\}\\
	\leq & K (\delta^a + (\delta^{a-r}+1)\max\{t_{m+k}^\alpha-t_{m}^\alpha,t_{m+k}-t_{m}\}).
\end{split}
\]
Choosing $\delta=\max\{t_{m+k}^\alpha-t_{m}^\alpha,t_{m+k}-t_{m}\}^{\frac{1}{r}}$, we obtain
\[
\|U^{m+k} - U^m\|_{L^\infty(\mathbb{R}^d)}\leq K (\max\{t_{m+k}^\alpha-t_{m}^\alpha,t_{m+k}-t_{m}\}^{\frac{a}{r}}+  \max\{t_{m+k}^\alpha-t_{m}^\alpha,t_{m+k}-t_{m}\} ).
\]

Let us denote $\eta=\max\{t_{m+k}^\alpha-t_{m}^\alpha,t_{m+k}-t_{m}\}$. If $\eta \leq 1$, then $\eta\leq \eta^{\frac{a}{r}}$. On the other hand, if $\eta>1$, then $\eta= \eta^{\frac{a}{r}}  \eta^{1-\frac{a}{r}} \leq K  \eta^{\frac{a}{r}}$ since $t_k\leq T$. Thus, the above estimate yields 
\[
\|U^{m+k} - U^m\|_{L^\infty(\mathbb{R}^d)}\leq K \max\{t_{m+k}^\alpha-t_{m}^\alpha,t_{m+k}-t_{m}\}^{\frac{a}{r}}.
\]

Using now the numerical identity in Lemma \ref{lem:tech}, we get
\[
\|U^{m+k} - U^m\|_{L^\infty(\mathbb{R}^d)}\leq K \max\{t_{m+k}^{\alpha-1} (t_{m+k}-t_{m}),t_{m+k}-t_{m}\}^{\frac{a}{r}}.
\]
From here, \eqref{equi:est2} follows directly. To prove \eqref{equi:est1} we just note that $t_{m+k}^{\alpha-1} \leq (t_{m+k}-t_m)^{\alpha-1} $, since $\alpha<1$. 
\end{proof}

\section{Compactness and convergence to viscosity solutions}\label{sec:CompConv}

In this section, we will show that numerical solutions have a uniform limit up to a subsequence. Moreover, we will show that the limit is actually a viscosity solution of \eqref{eq:main}, and satisfies a number of inherited properties. Moreover, whenever the uniqueness of viscosity solutions is known, we ensure full convergence of the scheme.

\subsection{Estimates for a continuous in time scheme}
Our aim is to apply Arzel\`a-Ascoli to ensure convergence of the scheme. To do so, we need to extend the concept of numerical solution to $\R^d\times[0,T]$. We do this by piecewise linear interpolation. More precisely, let $U^n$ satisfy \eqref{eq:NumSch} and,  for each $t\in[t_{k},t_{k+1}]$, define
\begin{equation}\label{eq:interptime}
U(x,t)= \frac{t_{k+1}-t}{\tau} U^{k}(x) + \frac{t-t_k}{\tau} U^{k+1}(x).
\end{equation}
We have the following equiboundedness and equicontinuity estimates in space. 

\begin{lem}\label{lem:equisspace}
Let the assumptions of Corollary \ref{cor:regspace} hold. Let also $U$ be defined by \eqref{eq:interptime}. Then,
\begin{enumerate}[\rm (a)]
	\item $U$ is  bounded uniformly in $\tau$ and $h$. More precisely,
	\[
	\|U(\cdot,t)\|_{L^\infty(\R^d)} \leq \|U^0\|_{L^\infty(\R^d)}, \quad \textup{for all} \quad t\in[0,T].
	\]
	\item $U$ is continuous in space uniformly  in $\tau$ and $h$. More precisely,
	\[
	\|U(\cdot+y,t)-U(\cdot,t)\|_{L^\infty(\R^d)} \leq C\max\{1, t^\alpha\} \left\{ \begin{split}
		 \max\{|y|^a, |y|^r\}, \quad  & \textup{if} \quad r\in (0,1],\\
		 \max\{|y|^\frac{a}{r},|y|\}, \quad  & \textup{if} \quad r\in (1,2], \quad
	\end{split}\right. \quad \textup{for all} \quad t\in[0,T],
	\]
	where $C=C(\alpha,r, T, \|D\|_{L^\infty(Q_T)},\|D\|_{X^r(Q_T)}, \|U^0\|_{L^\infty(\R^d)},\|U^0\|_{X^a(\R^d)})>0$.
\end{enumerate}
\end{lem}
\begin{proof}
For the boundedness result, we use the $L^\infty$-contraction result given in Lemma \ref{lem:contraction} with  $V^n=0$ (which is clearly a solution) to get
\[
\|U^k\|_{L^\infty(\R^d)}\leq \|U^0\|_{L^\infty(\R^d)},
\]
for all $k\geq0$. Then,
\[
\begin{split}
	\|U(\cdot,t)\|_{L^\infty(\R^d)} &\leq \frac{t_{k+1}-t}{\tau} \|U^{k}\|_{L^\infty(\R^d)} + \frac{t-t_k}{\tau} \|U^{k+1}\|_{L^\infty(\R^d)}\\
	&\leq \frac{t_{k+1}-t}{\tau} \|U^{0}\|_{L^\infty(\R^d)} + \frac{t-t_k}{\tau} \|U^{0}\|_{L^\infty(\R^d)}= \|U^{0}\|_{L^\infty(\R^d)}.
\end{split}
\]

To prove  uniform continuity in space, we adopt the notation
\[
f(y)=\left\{ \begin{split}
	\max\{|y|^a, |y|^r\},\quad  & \textup{if} \quad r\in (0,1],\\
	\max\{|y|^\frac{a}{r},|y|\},\quad  & \textup{if} \quad r\in (1,2]. \quad
\end{split}\right.
\]
Let $t\in [t_k,t_{k+1}]$. Direct computations using Corollary \ref{cor:regspace} show
\[
\begin{split}
	|U(x+y,t)-U(x,t)|\leq& \frac{t_{k+1}-t}{\tau} |U^{k}(x+y)-U^{k}(x)| + \frac{t-t_k}{\tau} |U^{k+1}(x+y)-U^{k+1}(x)|\\
	\leq&C f(y) \left( \frac{t_{k+1}-t}{\tau} \max\{1,t_k^\alpha\}  +  \frac{t-t_k}{\tau}  \max\{1,t_{k+1}^\alpha\}\right).
\end{split}
\]
If $t_{k+1}\leq 1$, the result follows directly. If $t_{k}\geq1$, the above estimate reads
\[
\begin{split}
	|U(x+y,t)-U(x,t)| &\leq C f(y) \left( \frac{t_{k+1}-t}{\tau}t_k^\alpha +  \frac{t-t_k}{\tau} t_{k+1}^\alpha\right)\\
	&=C f(y) \left(\left(1- \frac{t-t_k}{\tau}\right) t_k^{\alpha} + \frac{t-t_k}{\tau} t_{k+1}^\alpha\right)\\
	&=C f(y) \left(  t_k^{\alpha} +  \frac{t-t_k}{\tau}( t_{k+1}^\alpha- t_{k}^\alpha)\right).
\end{split}
\]
We can apply twice Lemma \ref{lem:tech} to get
\[
\begin{split}
	|U(x+y,t)-U(x,t)| & \leq C f(y) \left(  t_k^{\alpha} +  \frac{t-t_k}{\tau} t_{k+1}^{\alpha-1} (t_{k+1}-t_k) \right)\\
	&=C f(y) \left(  t_k^{\alpha} +  (t-t_k) t_{k+1}^{\alpha-1}  \right)\\
	&\leq C f(y) \left(  t^{\alpha} +  (t-t_k) t^{\alpha-1}  \right)\\
	&\leq C f(y) \left(  t^{\alpha} +  \frac{t^\alpha-t_k^\alpha}{\alpha}  \right)\\
	&\leq C f(y)  \left(1+ \frac{1}{\alpha}\right) t^{\alpha}.
\end{split}
\]
Finally, the case $t_k< 1<t_{k+1}$ follows in a similar way.
\end{proof}

We prove now equicontinuity in time. 

\begin{lem}\label{lem:estcontintime}
Let the assumptions of Corollary \ref{cor:EquicontinuityTime} hold. Let also $U$ be defined by \eqref{eq:interptime}. Then, for all $\tilde{t}>t\geq0$, we have
\begin{equation}\label{equi:est1cont}
	\|U(\cdot,\tilde{t}) - U(\cdot,t)\|_{L^\infty(\mathbb{R}^d)} \leq C \max\{(\tilde{t}-t)^{\alpha\frac{a}{r}}, (\tilde{t}-t)^{\frac{a}{r}}\},
\end{equation}	
and 
\begin{equation}\label{equi:est2cont}
	\|U(\cdot,\tilde{t}) - U(\cdot,t)\|_{L^\infty(\mathbb{R}^d)} \leq  C ((\tilde{t}^{\alpha-1}+1)(\tilde{t}-t) )^\frac{a}{r},
\end{equation}
where $C=C(\alpha,r, T, \|D\|_{L^\infty(Q_T)},\|D\|_{C_b([0,T]: X^r_b(\R^d))}, \|D\|_{C_b(\R^d: Lip([0,T]))}, \|U^0\|_{L^\infty(\R^d)},\|U^0\|_{X^a(\R^d)})>0$.
\end{lem}
\begin{proof} 
Let $k,j\in \mathbb{N}$ be such that $t\in[t_j,t_{j+1}]$ and $\tilde{t}\in [t_k,t_{k+1}]$. Assume first that $j=k$. Then, by definition and Corollary \ref{cor:EquicontinuityTime}, 
\begin{equation}
	\begin{split}
		|U(x,\tilde{t})- U(x,t)|&=\left|\left(\frac{t_{k+1}-\tilde{t}}{\tau}-\frac{t_{k+1}-t}{\tau}\right)U^{k}(x) +\left(\frac{\tilde{t}-t_k}{\tau}- \frac{t-t_k}{\tau} \right)U^{k+1}(x)\right| \\
		&= \frac{\tilde{t}-t}{\tau}|U^{k+1}(x)-U^{k}(x)|\\
		&\leq C  \frac{\tilde{t}-t}{\tau}  (\tau (t_{k+1}^{\alpha-1}+1))^{\frac{a}{r}}\\
		&\leq C  \left(\frac{\tilde{t}-t}{\tau}\right)^{\frac{a}{r}}  (\tau (t_{k+1}^{\alpha-1}+1))^{\frac{a}{r}}\\
		&\leq C\left( (t_{k+1}^{\alpha-1}+1)(\tilde{t}-t) \right)^{\frac{a}{r}}\\
		&\leq C \left( (\tilde{t}^{\alpha-1}+1)(\tilde{t}-t)\right)^{\frac{a}{r}}.
	\end{split}
\end{equation}
Now, if $k\geq j+1$, by the previous estimate, Corollary \ref{cor:EquicontinuityTime} and Lemma \ref{lem:tech},
\begin{equation}
	\begin{split}
		|U(x,t)-U(x,\tilde{t})|&\leq|U(x,t)-U^{j+1}(x)|+ |U(x,\tilde{t})-U^j(x)|+|U^{k}(x)-U^{j+1}(x)|\\
		&\leq  C\left( \left((t_{j+1}^{\alpha-1}+1)(t_{j+1}-t)\right)^{\frac{a}{r}}  +  \left((\tilde{t}^{\alpha-1}+1)(\tilde{t}-t_k)\right)^{\frac{a}{r}} +  \left((t_k^{\alpha-1}+1)(t_k-t_{j+1})\right)^{\frac{a}{r}}\right) \\
		&\leq  3C \left((\tilde{t}^{\alpha-1}+1)(\tilde{t}-t)\right)^{\frac{a}{r}},
	\end{split}
\end{equation}
which concludes the proof of the second estimate. The proof for the first one follows directly using $\tilde{t}^{\alpha-1} \leq (\tilde{t}-t)^{\alpha-1} $ since $\alpha<1$

\end{proof}

\subsection{Limit of the scheme}
We are now ready to prove that the solution of the numerical scheme has a limit, up to a subsequence.

\begin{cor}\label{coro:compactnessscheme}
Let the assumptions of Lemma \ref{lem:estcontintime} hold. There exists a function $u\in C_b(Q_T)$ and a subsequence $(h_j,\tau_j)\to0$ as $j\to \infty$ such that
\[
U\to u \quad \textup{locally uniformly in} \quad \mathbb{R}^d\times[0,T] \quad \textup{as} \quad j\to \infty. 
\]
Moreover, $u$ has the following properties:
\begin{enumerate}[\rm (a)]
\item \emph{(Boundedness)} $
\|u(\cdot,t)\|_{L^\infty(\R^d)} \leq \|U^0\|_{L^\infty(\R^d)}$  \textup{for all} $ t\in[0,T]$. 
\item \emph{(Space continuity)} $u$ is uniformly continuous in space and \[
\|u(\cdot+y,t)-u(\cdot,t)\|_{L^\infty(\R^d)} \leq C \max\{1, t^\alpha\} \left\{ \begin{split}
	 \max\{|y|^a, |y|^r\} \quad  & \textup{if} \quad r\in (0,1],\\
	 \max\{|y|^\frac{a}{r},|y|\} \quad  & \textup{if} \quad r\in (1,2], \quad
\end{split}\right. \quad \textup{for all} \quad t\in[0,T],
\]
where $C=C(\alpha,r, T, \|D\|_{L^\infty(Q_T)},\|D\|_{X^r(Q_T)}, \|U^0\|_{L^\infty(\R^d)},\|U^0\|_{X^a(\R^d)})>0$.
\item \emph{(Time continuity)} $u$ is uniformly continuous in time and, for all $\tilde{t}>t\geq0$, it satisfies
\begin{equation}\label{equi:est1contlimit}
	\|u(\cdot,\tilde{t}) - u(\cdot,t)\|_{L^\infty(\mathbb{R}^d)} \leq C \max\{(\tilde{t}-t)^{\alpha\frac{a}{r}}, (\tilde{t}-t)^{\frac{a}{r}}\},
\end{equation}	
and 
\begin{equation}\label{equi:est2contlimit}
	\|u(\cdot,\tilde{t}) - u(\cdot,t)\|_{L^\infty(\mathbb{R}^d)} \leq  C ((\tilde{t}^{\alpha-1}+1)(\tilde{t}-t) )^\frac{a}{r},
\end{equation}with $C=C(\alpha,r, T, \|D\|_{L^\infty(Q_T)}, \|D\|_{C_b([0,T]: X^r_b(\R^d))}, \|D\|_{C_b(\R^d: Lip([0,T]))}, \|U^0\|_{L^\infty(\R^d)},\|U^0\|_{X^a(\R^d)})>0$.
\end{enumerate}
\end{cor}
\begin{proof}
The existence of a locally uniform limit follows by Arzel\`a-Ascoli, since the solution of the scheme is equibounded and equicontinuous in $Q_T$ (Lemma \ref{lem:equisspace} and Lemma \ref{lem:estcontintime}). Also, the estimates are inherited by uniform convergence. 
\end{proof}

\subsection{Convergence to viscosity solutions}

Let us first introduce the concept of viscosity solutions. Recall that we are interested in the problem
\begin{equation}\label{eq:mainPDE}
\left\{
\begin{split}
	\partial_t^\alpha u(x,t)= D(x,t)\Le^{\mu,\sigma} [u](x,t), &  \quad \textup{for} \quad  (x,t) \in \mathbb{R}^d\times(0,T),\\
	u(x,0)=u_0(x),\hspace{2.25cm}  & \quad \textup{for} \quad x\in \mathbb{R}^d. 
\end{split}
\right.
\end{equation}
\begin{defi}\label{def:viscsol}

A function $u\in \textup{USC} (\R^d\times[0,T))$ (\textup{resp.} $u\in \textup{LSC} (\mathbb{R}^d\times[0,T))$) is a viscosity subsolution (\textup{resp.} supersolution) of \eqref{eq:mainPDE} if
\begin{enumerate}[\rm (a)]
	\item\label{def:viscsol-item1}  whenever $\phi\in C^2_b(\mathbb{R}^d\times[0,T))$ and $(x_0,t_0)\in\mathbb{R}^d\times(0,T) $ are such that $u-\phi$ attains a strict global maximum (\textup{resp.} minimum) at $(x_0,t_0)$,  then
	\[
	\begin{split}
		\partial_t^\alpha \phi(x_0,t_0)- D(x_0,t_0)\Le^{\mu,\sigma} [\phi](x_0,t_0) \leq& 0.\\
		(\textup{resp.} \geq&0) 
	\end{split}
	\]
	\item\label{def:viscsol-item2} $u(x,0)\leq u_0(x)$ (\textup{resp.} $u(x,0)\geq u_0(x)$).
\end{enumerate}
A function $u\in C_b(\mathbb{R}^d\times[0,T))$ is a viscosity solution of \eqref{eq:mainPDE} if it is both a viscosity subsolution and a viscosity supersolution.
\end{defi}

\begin{rem}\label{rem:Uniqueness}
Problem \eqref{eq:mainPDE} in the viscosity setting falls within the framework of the theory developed in \cite{topp2017existence} in the case where diffusivity does not explicitly depend on time, that is, $D=D(x)$. In that paper, the authors proved the existence and uniqueness of viscosity solutions to a general nonlinear parabolic problem with the Caputo derivative. Sufficient conditions for uniqueness of viscosity solutions for the local in space and fractional in time in the case where the coefficient can explicitly depend on time were given in \cite{namba2018existence}. Additional uniqueness results for weak solutions to similar equations can be found in \cite{allen2016parabolic,allen2017uniqueness}. 
\end{rem}

We have the following convergence result.

\begin{thm}\label{thm:convvisc}
Let the assumptions of Corollary \ref{coro:compactnessscheme} hold. Then the function $u$ given in Corollary \ref{coro:compactnessscheme} is a viscosity solution of \eqref{eq:mainPDE}. Furthermore, if viscosity solutions of \eqref{eq:mainPDE} are unique, the numerical scheme converges, that is,
\[
U\to u \quad \textup{locally uniformly in} \quad \R^d\times[0,T] \quad \textup{as} \quad (h,\tau)\to 0^+. 
\]  
\end{thm}

To prove the above result, we need to prove several intermediate steps. For technical reasons in the convergence proof, we will use a piecewise constant in time extension of the discrete-in-time scheme. More precisely, consider the function $V:\R^d\times[0,T)\to \R $ defined by
\begin{equation}
V(x,t) = U^k(x), \quad \textup{for all} \quad t\in [t_{k}, t_{k+1}).
\end{equation}
We will first prove that $V$ also converges locally uniformly to the same limit as $U$ given in Corollary \ref{coro:compactnessscheme}, up to the same subsequence.

\begin{lem}\label{lem:Vconv}
Let the assumptions of Corollary \ref{coro:compactnessscheme} hold. Let $u$ and $(h_j,\tau_j)$ be the sequence given in Corollary \ref{coro:compactnessscheme}.  Then, 
\[
V\to u \quad \textup{locally uniformly in} \quad \R^d\times[0,T] \quad \textup{as} \quad j\to \infty. 
\]
\end{lem}
\begin{proof}
First we note that
\[
|V(x,t)-u(x,t)|\leq |U(x,t)-u(x,t)|+ |U(x,t)-V(x,t)|.
\]
By Corollary \ref{coro:compactnessscheme}, the first term of the right-hand side converges locally uniformly to 0. On the other hand, if $t\in [t_k,t_{k+1})$, we can use the definition of interpolants and the uniform-in-time continuity in Corollary \ref{cor:EquicontinuityTime} to get
\[
\begin{split}
	|U(x,t)-V(x,t)|&=\left| \frac{t_{k+1}-t}{\tau} U^{k}(x) + \frac{t-t_k}{\tau} U^{k+1}(x)-U^{k}(x)\right| \\
	&=  \frac{t-t_k}{\tau}|U^{k+1}(x)- U^{k}(x)|\\
	&\leq |U^{k+1}(x)- U^{k}(x)|\\
	& \leq \widetilde{C}\tau^{\alpha \frac{a}{r}},
\end{split}
\]
which concludes also uniform convergence along the subsequence $(h_j,\tau_j)$.
\end{proof}

An important advantage of $V$ with respect to $U$ is the fact that it also satisfies the scheme  outside grid points for a modified version of $D$. We extend the notion of a discrete Caputo derivative \eqref{def:disccap} to hold for all times. More precisely, given $f:\R_+\to \R$, define the operator
\begin{equation}
\overline{\partial}^\alpha_\tau f(t)= \frac{\tau^{-\alpha}}{\Gamma(2-\alpha)} \left(f(t) - b_{n-1}f(t-t_n) - \sum_{k=1}^{n-1} (b_{n-k-1} - b_{n-k}) f(t-t_n+t_k)\right), \quad b_i := (i+1)^{1-\alpha} - i^{1-\alpha},
\end{equation}
defined for all $t\in[t_{n},t_{n+1})$ with $n\geq1$. Note that, if $t=t_n$ for some $n\geq 1$, then $\overline{\partial}^\alpha_\tau f(t_n)=\partial^\alpha_\tau f^n$. Let us also define 
\[
\overline{D}(x,t)=D^n(x), \quad \textup{for all} \quad t\in[t_{n},t_{n+1}).
\]

\begin{lem}\label{lem:eqV}Let $V$ be defined as above. Then, $V(x,t)=U^0(x)$ if $t\in[0,\tau)$, and
\[
\overline{\partial}^\alpha_\tau V(x,t)= \overline{D}(x,t)L_h V(x,t),
\]
for all $t\in [\tau, T)$.
\end{lem}
\begin{proof}
Clearly, $V(x,t)=U^0(x)$ if $t\in[0,\tau)$ by definition. Now let $t\in[t_n,t_{n+1})$ for some $n\geq1$. Using  directly the definition of $V$ and the scheme satisfied by $U$, we get
\[
\begin{split}
	\overline{\partial}^\alpha_\tau V(x,t)&=  \frac{\tau^{-\alpha}}{\Gamma(2-\alpha)} \left(V(x,t) - b_{n-1}V(x,t-t_n) - \sum_{k=1}^{n-1} (b_{n-k-1} - b_{n-k}) V(x,t-t_n+t_k)\right)\\
	&=\frac{\tau^{-\alpha}}{\Gamma(2-\alpha)} \left(U^n(x) - b_{n-1}U^0(x) - \sum_{k=1}^{n-1} (b_{n-k-1} - b_{n-k}) U^k(x)\right)\\
	&= \partial_\tau^\alpha U^n(x)\\
	&=D^n(x)L_h U^n(x)\\
	&= \overline{D}(x,t) \sum_{k=1}^\infty\left(U^n(x+y_k)+U^n(x-y_k)-2U^n(x)\right)\omega_k(h)\\
	&= \overline{D}(x,t) \sum_{k=1}^\infty\left(V(x+y_k,t)+V(x-y_k,t)-2V(x,t)\right)\omega_k(h)\\
	&= \overline{D}(x,t) L_h V(x,t).
\end{split}
\]
\end{proof}

We are now ready to prove Theorem \ref{thm:convvisc}. 

\begin{proof}[Proof of Theorem \ref{thm:convvisc}]
We will first prove that $u$ is a viscosity solution of \eqref{eq:mainPDE}. We do the argument only for subsolutions, since the one for supersolutions follows in the same way. First note that, by uniform convergence, 
\[
u(x,0)=\lim_{(h_j,\tau_j)\to0} U(x,0) = u_0(x),
\]
so the initial data is taken according to Definition \ref{def:viscsol} \eqref{def:viscsol-item2}. Now, let $(x_0,t_0)\in \R^d\times(0,T)$ and $\phi$ be as in the Definition \ref{def:viscsol} \eqref{def:viscsol-item1}. By uniform convergence (Lemma \ref{lem:Vconv}), we have that there exists a sequence $\{(x_\veps,t_\veps)\}$ such that
\begin{enumerate}[(i)]
	\item $V(x_\veps,t_\veps)-\phi(x_\veps,t_\veps)= \sup_{(x,t)\in \mathbb{R}^d\times[0,T]}(V(x,t)-\phi(x,t))=:M_\veps,$
	\item $V(x_\veps,t_\veps)-\phi(x_\veps,t_\veps)> V(x,t)-\phi(x,t),$ for all $(x,t)\in \mathbb{R}^d\times[0,T]\setminus(x_\veps,t_\veps)$, 
\end{enumerate}
and $(x_\veps,t_\veps)\to (x_0,t_0)$ as $\veps=(\tau,h)\to0$. Taking $\tau>0$ small enough such that $\{(x_\veps,t_\veps)\}\subset \R^d\times [\tau,T)$, we have, by Lemma \ref{lem:eqV}, that
\[
\overline{\partial}^\alpha_\tau V(x_\veps,t_\veps)=\overline{D}(x_\veps,t_\veps) L_h V(x_\veps,t_\veps).
\]
By basic properties of the operators $\overline{\partial}^\alpha_\tau$ and $L_h$ we have
\[
\overline{\partial}^\alpha_\tau (V-M_\veps)(x_\veps,t_\veps)=\overline{D}(x_\veps,t_\veps) L_h (V-M_\veps)(x_\veps,t_\veps).
\]
Let $t_n=n\tau$ be such that $t_\veps\in[t_n,t_{n+1}]$. Since $V(x_\veps,t_\veps)-M_\veps=\phi(x_\veps,t_\veps)$ and $V(x,t)-M_\veps\leq \phi(x,t)$, we then have
\[
\begin{split}
	&\overline{\partial}^\alpha_\tau (V-M_\veps)(x_\veps,t_\veps)\\
	&= \frac{\tau^{-\alpha}}{\Gamma(2-\alpha)} \left((V(x_\veps,t_\veps)-M_\veps) - b_{n-1}(V(x_\veps,t_\veps-t_n)-M_\veps) - \sum_{k=1}^{n-1} (b_{n-k-1} - b_{n-k}) (V(x_\veps,t_\veps-t_n+t_k)-M_\veps)\right)\\
	&= \frac{\tau^{-\alpha}}{\Gamma(2-\alpha)} \left(\phi(x_\veps,t_\veps) - b_{n-1}(V(x_\veps,t_\veps-t_n)-M_\veps) - \sum_{k=1}^{n-1} (b_{n-k-1} - b_{n-k}) (V(x_\veps,t_\veps-t_n+t_k)-M_\veps)\right)\\
	&\geq \frac{\tau^{-\alpha}}{\Gamma(2-\alpha)} \left(\phi(x_\veps,t_\veps) - b_{n-1}\phi(x_\veps,t_\veps-t_n) - \sum_{k=1}^{n-1} (b_{n-k-1} - b_{n-k}) \phi(x_\veps,t_\veps-t_n+t_k)\right)\\
	&= \overline{\partial}^\alpha_\tau \phi(x_\veps,t_\veps).
\end{split}
\]
On the other hand,
\[
\begin{split}
	L_h (V-M_\veps)(x_\veps,t_\veps)&=\sum_{k=1}^\infty\left((V-M_\veps)(x_\veps+y_k, t_\veps)+(V-M_\veps)(x_\veps-y_k,t_\veps)-2(V-M_\veps)(x_\veps,t_\veps)\right)\omega_k(h)\\
	&=\sum_{k=1}^\infty\left((V-M_\veps)(x_\veps+y_k, t_\veps)+(V-M_\veps)(x_\veps-y_k,t_\veps)-2\phi(x_\veps,t_\veps)\right)\omega_k(h)\\
	&\leq\sum_{k=1}^\infty\left(\phi(x_\veps+y_k, t_\veps)+\phi(x_\veps-y_k,t_\veps)-2\phi(x_\veps,t_\veps)\right)\omega_k(h)\\
	&= L_h \phi (x_\veps,t_\veps).
\end{split}
\]
Combining the above inequalities,
\[
\overline{\partial}^\alpha_\tau \phi(x_\veps,t_\veps)\leq \overline{D}(x_\veps,t_\veps) L_h \phi (x_\veps,t_\veps).
\]
Finally, we rewrite it as
\[
\overline{\partial}^\alpha_\tau \phi(x_\veps,t_\veps)\leq D(x_\veps,t_\veps)  L_h \phi (x_\veps,t_\veps) +(\overline{D}(x_\veps,t_\veps)-D(x_\veps,t_\veps))L_h \phi (x_\veps,t_\veps).
\]
Note that, by regularity of $\phi$ and Remark \ref{rem:boundwei},
\[
|\overline{D}(x_\veps,t_\veps)-D(x_\veps,t_\veps)||L_h \phi (x_\veps,t_\veps) |\leq C \sup_{(x,t)\in \R^d\times[0,T]} |\overline{D}(x,t)-D(x,t)|\to 0 \quad \textup{as} \quad h,\tau\to0.
\]
We use now consistency of both the space and time discretizations to get
\[
\partial_t^\alpha \phi(x_\veps,t_\veps)\leq D(x_\veps,t_\veps)  \Le^{\mu,\sigma}  \phi (x_\veps,t_\veps) + o_\veps(1),
\]
and taking limits as $\veps\to0$, using the regularity of $\phi$ and continuity of $\Le^{\mu,\sigma}$ and $\partial_t^\alpha$ we get
\[
\partial_t^\alpha \phi(x_0,t_0)\leq D(x_0,t_0)  \Le^{\mu,\sigma}  \phi (x_0,t_0).
\]
This concludes the proof that $u$ is a viscosity subsolution.

Finally, if the uniqueness of viscosity solutions holds, the whole sequence $U$ converges, and the result obtained is the result stated.
\end{proof}

\subsection{Convergence to classical solutions}
In some cases, we are able to prove that $u$, the limit of solutions of the numerical scheme, is actually a classical solution. This only works in the purely nonlocal in space case, and we will just write $\Le^{\mu}$ instead of $\Le^{\mu,\sigma}$ since the local part will not be present.

\begin{thm}\label{thm:ConvergencePointwise}
Assume \eqref{As:wepos}, \eqref{As:webdd}, \eqref{As:constrong}, \eqref{As:Le} for some $r\in[0,1)$ and $h,\tau>0$. Let $u_0\in X^a_b(\R^d)$ for some $a\in(r,1]$, $D\in C_b([0,T]: X^r_b(\R^d))\cap C_b(\R^d: Lip([0,T]))$ nonnegative and $U^n$ satisfy \eqref{eq:NumSch}. Then, the limit $u$ given in Corollary \ref{coro:compactnessscheme} satisfies \eqref{eq:mainPDE} in the pointwise sense.
\end{thm}

\begin{proof}
The initial condition $u(x,0)=u_0(x)$ is satisfied trivially by uniform convergence. For notational convenience, in this proof we write $U_h$ to state the dependence of $U$ on the numerical parameters.

We now prove that $L_h U\to \Le^{\mu} u$ pointwise. Note that, by Corollary \ref{coro:compactnessscheme}, we have that $u(\cdot,t)\in X^a_b(\R^d)$ uniformly in $t$, so that $\Le^\mu u$ is defined pointwise by \eqref{As:Le}, since $a>r$. In the same way, by Lemma \ref{lem:equisspace}, $U_h(\cdot,t)\in X^a_b(\R^d)$ uniformly in $h$ and $t$, and thus $\Le^\mu U_h$ is again defined pointwise by \eqref{As:Le}, since $a>r$.

First, we split
\[
|L_h U_h(x,t) - \Le^\mu u(x,t)| \leq |L_h U_h(x,t) - \Le^\mu U_h(x,t)| + |\Le^\mu U_h(x,t) - \Le^\mu u(x,t)| = I+ II.
\]
On one hand, by \eqref{As:constrong} and Lemma  \ref{lem:equisspace},
\[
I \leq \|U_h(\cdot,t)\|_{X^a_b(\R^d)} h^{r-a} \leq C h^{r-a} \to 0 \quad \textup{as} \quad h\to0. 
\]
Now, denoting $W_h=U_h-u$, we have
\[
II\leq\int_{|z|>0}  \left|W_h(x+z,t)+W_h(x-z,t)-2 W_h(x,t)\right| \dd \mu(z),
\]
Note that, by the regularity of $U_h$ and $u$, we have
\[
|W_h(x+z,t)+W_h(x+z,t)-2 W_h(x,t)| \leq \min\{|z|^a,1\}.
\]
Moreover, $W_h\to0$ as $h\to0$ pointwise. Thus, by \eqref{As:Le} and the dominated convergence theorem, $II\to0$ as $h\to0$.

The argument for the time fractional derivative follows in a similar way using Proposition \ref{prop:L1Truncation} (stated and proved later) and time regularity and we omit it.
\end{proof}


\section{Orders of convergence for classical solutions}\label{sec:Orders}
We now proceed to finding the estimates of the convergence error for the numerical scheme \eqref{eq:NumSch}. We want to stress that, in order to prove orders of convergence, we will only assume the regularity of solutions that we have proved in Corollary \ref{coro:compactnessscheme} together with the fact that viscosity solutions are, in some cases, classical solutions (Theorem \ref{thm:ConvergencePointwise}). In this section, we will always assume that $u$ is the classical solution of  \eqref{eq:mainPDE} given by Theorem \ref{thm:ConvergencePointwise}.

Let us define the error as the difference between the exact solution of \eqref{eq:main} and its numerical approximation \eqref{eq:NumSch}, that is, 
\[
	e^n(x) := u(x,t_n) - U^n(x).
\]
Then it is straightforward to show the following estimate.
\begin{lem}\label{lem:NumError}
Assume \eqref{As:wepos}, \eqref{As:webdd}, \eqref{As:constrong}, \eqref{As:Le} for some $r\in[0,1)$ and $h,\tau>0$. Let $u_0\in X^a_b(\R^d)$ for some $a\in(r,1]$, $D\in C_b([0,T]: X^r_b(\R^d))\cap C_b(\R^d: Lip([0,T]))$ nonnegative and $U^n$ satisfy \eqref{eq:NumSch}. 
Define the truncation errors by
\[
R^n_x(x) := \mathcal{L}^{\mu,\sigma} u(x,t_n) - L_h u(x,t_n), \quad \textup{and} \quad R^n_t(x) := \partial^\alpha_t u(x,t_n) - \partial^\alpha_\tau u(x,t_n),
\]
and assume that there exist positive numbers $e_x^n$ and $e_t^n$, independent on $n$, such that
\begin{equation}\label{eqn:ErrorsEF}
	\frac{\tau^\alpha}{\Gamma(\alpha)} \sum_{k=0}^{n-1} (n-k)^{\alpha-1} \|R_{x}^{k+1}\|_{L^\infty(\mathbb{R}^d)} \leq e_{x}^{n}, \quad \textup{and} \quad \frac{\tau^\alpha}{\Gamma(\alpha)} \sum_{k=0}^{n-1} (n-k)^{\alpha-1} \|R_{t}^{k+1}\|_{L^\infty(\mathbb{R}^d)} \leq e_{t}^{n}.
\end{equation} 
Then, we have
\begin{equation}\label{eq:NumError}
	\begin{split}
		\|e^n\|_{L^\infty(\mathbb{R}^d)} \leq C\left(e_{x}^{n} + e_{t}^{n}\right),
	\end{split}
\end{equation}
with $C=C(\alpha,r, T, \|D\|_{L^\infty(Q_T)}, \|D\|_{C_b([0,T]: X^r_b(\R^d))}, \|D\|_{C_b(\R^d: Lip([0,T]))}, \|U^0\|_{L^\infty(\R^d)},\|U^0\|_{X^a(\R^d)})>0$.
\end{lem}
\begin{proof}
By plugging the definition of the error into the main equation \eqref{eq:main} we can show that
\[
\partial^\alpha_\tau e^n(x) = D^n(x) L_h e^n(x) + R^n_x(x) + R^n_t(x),
\]
where $R_{x}^n$ and $R^n_t$ are the truncation errors for space and time, respectively. By a similar $\esssup$ and $\essinf$ argument as in the proof of Lemma \ref{lem:decLh} we can show that
\begin{equation}\label{eq:ErrorEstimate}
	\partial^\alpha_\tau \|e^n\|_{L^\infty(\mathbb{R}^d)} \leq \|R^n_x\|_{L^\infty(\mathbb{R}^d)} + \|R^n_t\|_{L^\infty(\mathbb{R}^d)}.
\end{equation}
Application of the Gr\"onwall inequality \eqref{eq:GrongCap1} with $y^n=\|e^n\|_{L^\infty(\R^d)}$, $\lambda_{0}=\lambda_1 = 0$ and $F^n=\|R^n_x\|_{L^\infty(\mathbb{R}^d)} + \|R^n_t\|_{L^\infty(\mathbb{R}^d)}$ (i.e. $F_1 = e_{x}+e_t$, and $F_2 = 0$) yields the result since the numerical scheme and the equation have the same initial condition, that is, $y^0=\|u_0-u_0\|_{L^\infty(\R^d)}=0$. 
\end{proof}
As can be seen from the above result, the error of the numerical method decomposes in two terms: $e_x$ for space error and $e_t$ for the time error.  This errors are obtained as bounds of (discrete) fractional integrals of the truncation errors as seen in \eqref{eqn:ErrorsEF}. We will split our further discussion into two parts regarding time and space errors separately.

\subsection{Error in space}
The bound for the error in space is straightforward and follows from the assumption \eqref{As:constrong} and the proved regularity for $u$.

\begin{prop}
Let the assumptions and notations of Lemma \ref{lem:NumError} hold. Then, $e_x$ given in \eqref{eqn:ErrorsEF} can be chosen as follows:
\begin{equation}
	e_x^{n} = Ct_n^\alpha h^{a-r}, 
\end{equation}
where $C=C(\alpha, r, T, \|D\|_{L^\infty(Q_T)}, \|D\|_{C_b([0,T]: X^r_b(\R^d))}, \|D\|_{C_b(\R^d: Lip([0,T]))}, \|U^0\|_{L^\infty(\R^d)},\|U^0\|_{X^a(\R^d)})>0$. 
\end{prop}
\begin{proof}
First note that  $u\in X^a_b(\R^d)$ by Corollary \ref{coro:compactnessscheme}.  Thus, by \eqref{As:constrong}, we have
\[
	\|R^n_x\|_{L^\infty(\mathbb{R}^d)}=\|\mathcal{L}^{\mu,\sigma} u(x,t_n) - L_h u(x,t_n)\|_{L^\infty(\mathbb{R}^d)} =  \|u(\cdot,t_n)\|_{X^a_b(\R^d)} O(h^{a-r}) \leq C h^{a-r}.
\] 
From here, we can estimate $e_x$ in \eqref{eqn:ErrorsEF}. More precisely,
\begin{equation}
	\frac{\tau^\alpha}{\Gamma(\alpha)} \sum_{k=0}^{n-1} (n-k)^{\alpha-1} \|R^n_x\|_{L^\infty(\mathbb{R}^d)} \leq C h^{a-r} \frac{\tau^\alpha}{\Gamma(\alpha)} \sum_{k=1}^{n} k^{\alpha-1},
\end{equation}
where we have changed the order of summation. Now, bounding the sum by the integral we obtain
\begin{equation}
	\frac{\tau^\alpha}{\Gamma(\alpha)} \sum_{k=0}^{n-1} (n-k)^{\alpha-1} \|R^n_x\|_{L^\infty(\mathbb{R}^d)} \leq C h^{a-r} \frac{\tau^\alpha}{\Gamma(\alpha)} \int_0^n x^{\alpha-1} \dd x = \frac{C}{\Gamma(1+\alpha)} h^{a-r} \tau^{\alpha} n^\alpha.
\end{equation}
Using the fact that $(n\tau)^\alpha = t_n^\alpha$ ends the proof.
\end{proof}

\subsection{Error in time}
The truncation error for the L1 scheme for the Caputo derivative has previously been investigated in \cite{kopteva2019error, stynes2017error}. In these works, the authors assume certain estimates of regularity of the solution, in particular 
\[
|\partial^m_t u(x,t)| \leq C(1+t^{\alpha-m}) \quad \text{for} \quad m = 1,2.
\]
These, although natural for the subdiffusion, have to be \textit{assumed a priori}. On the other hand, in our above results we actually prove that the regularity estimate is always satisfied with $m=1$ (in a slightly weaker form). The following result gives the time truncation error for any function that satisfies the regularity estimate that we have obtained from our numerical scheme.

\begin{prop}\label{prop:L1Truncation}
Let $\alpha \in (0,1)$ and $\tau\in(0,1)$. Let also $y=y(t)$ be a continuous function satisfying 
\[|y(t)-y(s)| \leq \widetilde{C}(1+\max\{t,s\}^{\alpha-1})|t-s| \quad \textup{for all} \quad s,t\in[0,T].
\] 
for some $\widetilde{C}>0$. Then, we have
\begin{equation}
	|\partial^\alpha_t y(t_n) - \partial^\alpha_\tau y(t_n)| \leq C 
	\begin{cases}
		1, & n = 1, \\
		t_{n-1}^{\alpha-1}\tau^{1-\alpha}, & n > 1,
	\end{cases}
\end{equation}
where $t_n=\tau n \in [0,T]$ and $C=C(\alpha, \widetilde{C})$.
\end{prop}
\begin{proof}
Let $I_\tau [y]$ be the piecewise linear interpolation function of $y$ with nodes $t_k=\tau k$ for $k=0,...,n$. We recall that the $L^1$ discretization $\partial_{\tau}^\alpha$ of $\partial_{t}^\alpha$ is equivalently defined by
\[
\partial_{\tau}^\alpha y(t_n)= \partial_{t}^\alpha (I_\tau[y])(t_n).
\]
Then, 
\[
r^n_t:= \partial^\alpha_t y(t_n) - \partial^\alpha_\tau y(t_n) =  \partial^\alpha_t y(t_n) - \partial_{t}^\alpha (I_\tau[y])(t_n).
\]
We use the definition of the Caputo derivative given in \eqref{eq:capdef}  to get 
\[
\begin{split}
	r^n_t 
	&= \frac{1}{\Gamma(1-\alpha)}\left( \frac{y(t_n)-y(0)-I_\tau [y](t_n)+I_\tau [y](0)}{t^{\alpha}}+ \alpha \int_0^{t_n} \frac{y(t_n)-y(s)-I_\tau [y](t_n) + I_\tau [y](s)}{ (t_n-s)^{\alpha+1}}\dd s\right)\\
	&=  \frac{\alpha}{\Gamma(1-\alpha)} \int_0^{t_n} (t_n-s)^{-\alpha-1} (I_\tau [y](s)-y(s))\dd s,
\end{split}
\] 
where we have used that $I_\tau [y](t_k) = y(t_k)$ for all $k=0,...,n$. From here we proceed in several steps.

\textbf{Step 1:} We estimate first the interpolation error. Note that we cannot use the classical second order error estimate since it requires twice-differentiability of $y$, which is not true here. The error at $s=t_k$ for $k=0,\ldots,n$ is zero since the interpolation is exact there. Now let $s\in (t_k, t_{k+1})$ for some $k=0,\ldots,n$. Then
\[
\begin{split}
	|I_\tau [y](s)-y(s)| 
	&= \left| \frac{s-t_k}{\tau} y(t_{k+1}) + \frac{t_{k+1}-s}{\tau} y(t_k) - y(s)\right| \\
	&\leq \frac{s-t_k}{\tau} |y(t_{k+1}) - y(s)| + \frac{t_{k+1}-s}{\tau} |y(s) - y(t_k)|,
\end{split}
\]
since the coefficients add up to $1$. Next, using our regularity assumption, we have 
\begin{equation}\label{eqn:InterpolationError}
	\begin{split}
		|I_\tau y(s)-y(s)| 
		&\leq \widetilde{C}\left(\frac{(s-t_k)(t_{k+1}-s)}{\tau} (1+ t_{k+1}^{\alpha-1}) +  \frac{(s-t_k)(t_{k+1}-s)}{\tau} (1 + s^{\alpha-1})\right) \\
		&\leq 2\widetilde{C} \frac{(s-t_k)(t_{k+1}-s)}{\tau} (1+s^{\alpha-1}),
	\end{split}
\end{equation}
where the last inequality follows from the fact that $t_{k+1} \geq s$ and $\alpha-1< 0$. 

\textbf{Step 2:} We estimate now $r_t^n$ when $n=1$. Using (\ref{eqn:InterpolationError}) we have, for $\tau<1$, that
\[
\begin{split}
|r^1_t| &\leq \frac{2\widetilde{C}\alpha}{\Gamma(1-\alpha)} \tau^{-1} \int_0^\tau (\tau-s)^{-\alpha-1} s(\tau-s) (1+s^{\alpha-1}) \dd s\\
&\leq \frac{2\widetilde{C}\alpha}{\Gamma(1-\alpha)} \frac{\tau + \tau^\alpha}{\tau} \int_0^\tau (\tau-s)^{-\alpha} \dd s\\
& \leq  \frac{2\widetilde{C}\alpha}{\Gamma(2-\alpha)} (\tau^{1-\alpha} + 1) \leq \frac{4\widetilde{C}\alpha}{\Gamma(2-\alpha)},
\end{split}
\]
where we have used $\tau^{1-\alpha} \leq 1$.
	
\textbf{Step 3: }Now we prove the result for $n>1$. The integral in the remainder can be split onto each subinterval as follows: 
\begin{equation}\label{eqn:TruncationRemainder}
	|r^n_t|\leq \frac{\alpha}{\Gamma(1-\alpha)} \sum_{k=0}^{n-1}\mathcal{I}_k \quad \textup{with} \quad  \mathcal{I}_k= \int_{t_k}^{t_{k+1}} (t_n-s)^{-\alpha-1} |I_\tau [y](s) - y(s)| \dd s. 
\end{equation}
The estimates for $k=0$ and $k=n-1$ have to be done separately using (\ref{eqn:InterpolationError}). Hence, by a similar reasoning as above,
\begin{equation}\label{eqn:Truncation(k=0)}
	\begin{split}
		\mathcal{I}_0&\leq 2\widetilde{C} \tau^{-1} \int_0^\tau (t_n-s)^{-\alpha-1} (\tau-s) (s+s^{\alpha}) \dd s \\
		&\leq  4\widetilde{C} (t_{n}-\tau)^{-\alpha-1}\int_0^\tau s^\alpha \dd s\\
		& = \frac{4\widetilde{C}}{(1+\alpha)} \frac{1}{(n-1)^{1+\alpha}}\leq  \frac{4\widetilde{C}}{(1+\alpha)} (n-1)^{\alpha-1} =  \frac{4\widetilde{C}}{(1+\alpha)} t_{n-1}^{\alpha-1} \tau^{1-\alpha},
	\end{split}
\end{equation}
and
\begin{equation}\label{eqn:Truncation(k=n-1)}
	\begin{split}
		\mathcal{I}_{n-1} &\leq2 \widetilde{C} \tau^{-1}\int_{t_{n-1}}^{t_{n}} (t_n-s)^{-\alpha} (s-t_{n-1}) (1+s^{\alpha-1}) \dd s \\
		&\leq 2\widetilde{C} (1+t_{n-1}^{\alpha-1}) \int_{t_{n-1}}^{t_{n}} (t_n-s)^{-\alpha}\dd s \\
		&= \frac{2\widetilde{C}}{1-\alpha}  (1+t_{n-1}^{\alpha-1})\tau^{1-\alpha} \leq \frac{4 T^{1-\alpha}\widetilde{C}}{1-\alpha}t_{n-1}^{\alpha-1}\tau^{1-\alpha}.
	\end{split}
\end{equation}
Thus, we are left with estimating all remaining terms $k=1,...,n-2$ with $n>2$. To this end, we can write
\[
\begin{split}
	\sum_{k=1}^{n-2}  \mathcal{I}_k  &\leq 2 \widetilde{C} \tau \sum_{k=1}^{n-2}  \int_{t_k}^{t_{k+1}} (t_n-s)^{-\alpha-1} (1+s^{\alpha-1}) \dd s\\
	& = 2 \widetilde{C} \tau \int_{\tau}^{t_{n-1}} (t_n-s)^{-\alpha-1} \dd s + 2 \widetilde{C} \tau \int_{\tau}^{t_{n-1}} (t_n-s)^{-\alpha-1} s^{\alpha-1}ds  =  2\tilde{C} (S_1+S_2),
\end{split}
\]
where we have estimated $(s-t_k)(t_{k+1}-s) \leq \tau^2$. For the first term, we solve the integral explicitly and estimate as follows:
\[
S_1=\tau \int_{\tau}^{t_{n-1}} (t_n-s)^{-\alpha-1} \dd s \leq  \frac{\tau}{\alpha (t_{n}-t_{n-1})^\alpha} = \frac{\tau^{1-\alpha} }{\alpha} t_{n-1}^{\alpha-1} t_{n-1}^{1-\alpha} \leq \frac{T^{1-\alpha}}{\alpha} t_{n-1}^{\alpha-1} \tau^{1-\alpha}  .
\]
To estimate $S_2$, we perform the change of variables $s=t_n \rho$, and note that   for $n\geq2$, we have that $1/n \leq 1/2 \leq (n-1)/n$. Thus
\[
\begin{split}
S_2&= \frac{1}{n} \int_{\frac{1}{n}}^{\frac{n-1}{n}} (1-\rho)^{-\alpha-1} \rho^{\alpha-1} \dd \rho\\
& =  \frac{1}{n} \int_{\frac{1}{n}}^{\frac{1}{2}} (1-\rho)^{-\alpha-1} \rho^{\alpha-1} \dd\rho  +  \frac{1}{n} \int_{\frac{1}{2}}^{\frac{n-1}{n}} (1-\rho)^{-\alpha-1} \rho^{\alpha-1} \dd\rho\\
&\leq \frac{2^{\alpha+1}}{n} \int_{\frac{1}{n}}^{\frac{1}{2}} \rho^{\alpha-1} \dd\rho  +  \frac{2^{1-\alpha}}{n}  \int_{\frac{1}{2}}^{\frac{n-1}{n}} (1-\rho)^{-\alpha-1} \dd\rho\\
&\leq  \frac{C_1}{n} + \frac{C_2}{n} \left(1- \frac{n-1}{n}\right)^{-\alpha}.
\end{split}
\]
From here, it is easy to see that $S_2 \leq C  t_{n-1}^{\alpha-1}  \tau^{1-\alpha}$. This concludes the proof.
\end{proof}

Having the bound for the truncation error in time  we can give a bound for the error in time of the numerical scheme.
\begin{cor}\label{cor:TimeOrder}
	Let the assumptions and notations of Lemma \ref{lem:NumError} hold. Then, $e_t^n$ given in \eqref{eqn:ErrorsEF} can be chosen as follows:
\begin{equation}
	e_t^n = C t_n^{2\alpha-1} \tau^{1-\alpha},
\end{equation} 
where $C=C(\alpha, r, T, \|D\|_{L^\infty(Q_T)}, \|D\|_{C_b([0,T]: X^r_b(\R^d))}, \|D\|_{C_b(\R^d: Lip([0,T]))}, \|U^0\|_{L^\infty(\R^d)},\|U^0\|_{X^a(\R^d)})>0$.
\end{cor}
\begin{proof}
Note that, by Corollary \ref{coro:compactnessscheme}, we have that 
\[
	\|u(\cdot,t) - u(\cdot,s)\|_{L^\infty(\R^d)} \leq  \tilde{C} (1+\max\{t,s\}^{\alpha-1})|t-s |.
\]
Thus, by Proposition \ref{prop:L1Truncation}, we have
\begin{equation}
	\|R_t^n\|_{L^\infty(\R^d)}=\|\partial^\alpha_t u(\cdot,t_n) - \partial^\alpha_\tau u(\cdot,t_n)\|_{L^\infty(\R^d)} \leq C 
	\begin{cases}
		1, & n = 1, \\
		t_{n-1}^{\alpha-1}\tau^{1-\alpha}, & n > 1.
	\end{cases}
\end{equation}
Coming back to \eqref{eqn:ErrorsEF} we can then estimate
\begin{equation}
	\begin{split}
		\frac{\tau^\alpha}{\Gamma(\alpha)}\sum_{k=0}^{n-1} &(n-k)^{\alpha-1} \|R^{k+1}_t\|_{L^\infty(\mathbb{R}^d)} \leq C\tau^\alpha \left(n^{\alpha-1} + \sum_{k=1}^{n-1} (n-k)^{\alpha-1} k^{\alpha-1}\right) \\
		&= C n^{2\alpha-1} \tau^\alpha\left(\frac{1}{n^\alpha}+\frac{1}{n}\sum_{k=1}^{n-1} \left(1-\frac{k}{n}\right)^{\alpha-1} \left(\frac{k}{n}\right)^{\alpha-1}\right). 
	\end{split}
\end{equation}
The sum above is the Riemann sum for the beta function $B(\alpha,\alpha)$, and therefore, the parentheses are bounded with respect to $n$. Renaming the constant leads to
\begin{equation}
	\frac{\tau^\alpha}{\Gamma(\alpha)}\sum_{k=0}^{n-1} (n-k)^{\alpha-1} \|R^{k+1}_t\|_{L^\infty(\mathbb{R}^d)} \leq C n^{2\alpha-1}\tau^\alpha = C t_n^{2\alpha-1} \tau^{1-\alpha}, 
\end{equation}
what concludes the proof.
\end{proof}
Therefore, the error in time depends on the range of $\alpha$ and is different for small and large times. More specifically, when $\alpha \in [1/2, 1)$ the order is globally in time equal to $1-\alpha$ since $t_n^{2\alpha-1}\tau^{1-\alpha} \leq T^{2\alpha-1}\tau^{1-\alpha}$. When $\alpha\in (0,1/2)$ the scheme is again accurate with order $1-\alpha$ but only away from $t=0$. Near the origin we have the order deteriorates to $\alpha$ as the first step yields $t_1^{2\alpha-1} \tau^{1-\alpha} = \tau^\alpha$. This loss of accuracy near $t=0$ is consistent with the literature findings in the case of local in space linear subdiffusion equation \cite{stynes2017error}. However, in contrast to these results, we do not require any a priori regularity assumptions. 

\section{Numerical experiments}	
In this section we present several numerical verifications of the devised numerical scheme \eqref{eq:NumSch}. Our implementation has been done in the Julia programming language with the use of multithreading. For example, the assembly of the spatial operator matrix $L_h$ can be performed in parallel, as well as various error computations. Where necessary, all integrals have been calculated using Gaussian quadrature. For illustration purposes, we restrict our attention to one spatial dimension. In each example concerning equations in the whole space we have truncated the domain to $[-X, X]$ with $X>0$ chosen sufficiently large for the truncation error to be negligible for our investigations of the discretization error. 

\subsection{Numerical experiments with exact solutions}
In some rare cases, an exact solution to the main problem \eqref{eq:main} can be found, and the numerical solution can be compared with it. For this purpose let us consider the space-time fractional diffusion equation
\begin{equation}\label{eq:FracDiff}
\partial^\alpha_t u + D(x,t)(-\Delta)^s u = 0, \quad x\in\mathbb{R}, \quad t\in(0,T), \quad \alpha\in(0,1), \quad s\in(0,1),
\end{equation}
where the fractional Laplacian is given by 
\begin{equation}
-(-\Delta)^s u(x) = k_s \int_\mathbb{R} \frac{u(x+z) + u(x-z) - 2 u(x)}{|x-z|^{1+2s}} \dd z,
\end{equation}
where $k_s$ is a known constant. For the discretization of the fractional Laplacian \eqref{eq:SpaceOperatorDiscrete}, we use the weights originating from taking powers of the discrete Laplacian (see \cite{ciaurri2015connection}) and have been proved to be of second order accuracy in \cite{dTEnJa18}. The precise expression of the weights is the following: 
\begin{equation}\label{eq:LaplacianWeights}
\omega_k(h) = \frac{(-1)^{k+1}}{2\pi h^{2s}} \int_0^{2\pi} \left(4\left(\cos\frac{x}{2}\right)^2\right)^s e^{i k x} \dd x = \frac{1}{h^{2s}} \frac{2^{2s}\Gamma(\frac{1}{2}+s)\Gamma(k-s)}{\sqrt{\pi}|\Gamma(-s)|\Gamma(k+1+s)}.
\end{equation}
When implementing the above formula, we have tested two methods. One is the evaluation of the Fourier integral by means of Gaussian quadrature, while the other is done by using the exact expression in terms of gamma functions. In the latter case, one must keep in mind to use the asymptotic form of the gamma function for large values of $k$. These two approaches gave satisfactory results when implemented with the use of parallelism. 

\subsubsection{Local-in-time problem with constant diffusivity}
Consider the fractional diffusion with constant diffusivity, say $D(x,t) = 1$. In that case, the exact solution is (\cite[formulas (3.2) and (3.19)]{mainardi2007fundamental}) 
\begin{equation}\label{eq:ExactSuperDiff}
u(x,t) = \int_\mathbb{R} G(x-z, t) u_0(z) \dd z,
\end{equation}
where the Green's function is given by
\begin{equation}\label{eq:ExactSuperDiffGreen}
G(x,t) = (2\pi t^{\frac{\alpha}{2s}})^{-1} \int_\mathbb{R} E_\alpha(-|\xi|^{2s}) e^{-i \xi x t^{-\frac{\alpha}{2s}}} \dd\xi.  
\end{equation}
For the local-in-time diffusion, that is $\alpha = 1$, the convolution can be evaluated explicitly when $u_0(x) = G(x,1)$ yielding
\begin{equation}\label{eq:FracDiffExactSuperDiff}
u(x,t) = G(x,t+1). 
\end{equation}
The results of the calculations are presented in Fig. \ref{fig:ErrorXExactSuperDiff}. The time discretization parameter has had to be chosen sufficiently small in order to extract the information about the sole influence of space discretization. We have taken $\tau = 2^{-12}$ with $T=1$, but other choices were also explored. As can be seen, numerical results are consistent with the second-order accuracy in space of the chosen discretization. Therefore, we can conclude that, due to high regularity of the solution, the scheme's empirical order reproduces the truncation error estimate. 

\begin{figure}[h!]
\centering
\includegraphics[scale = 0.6]{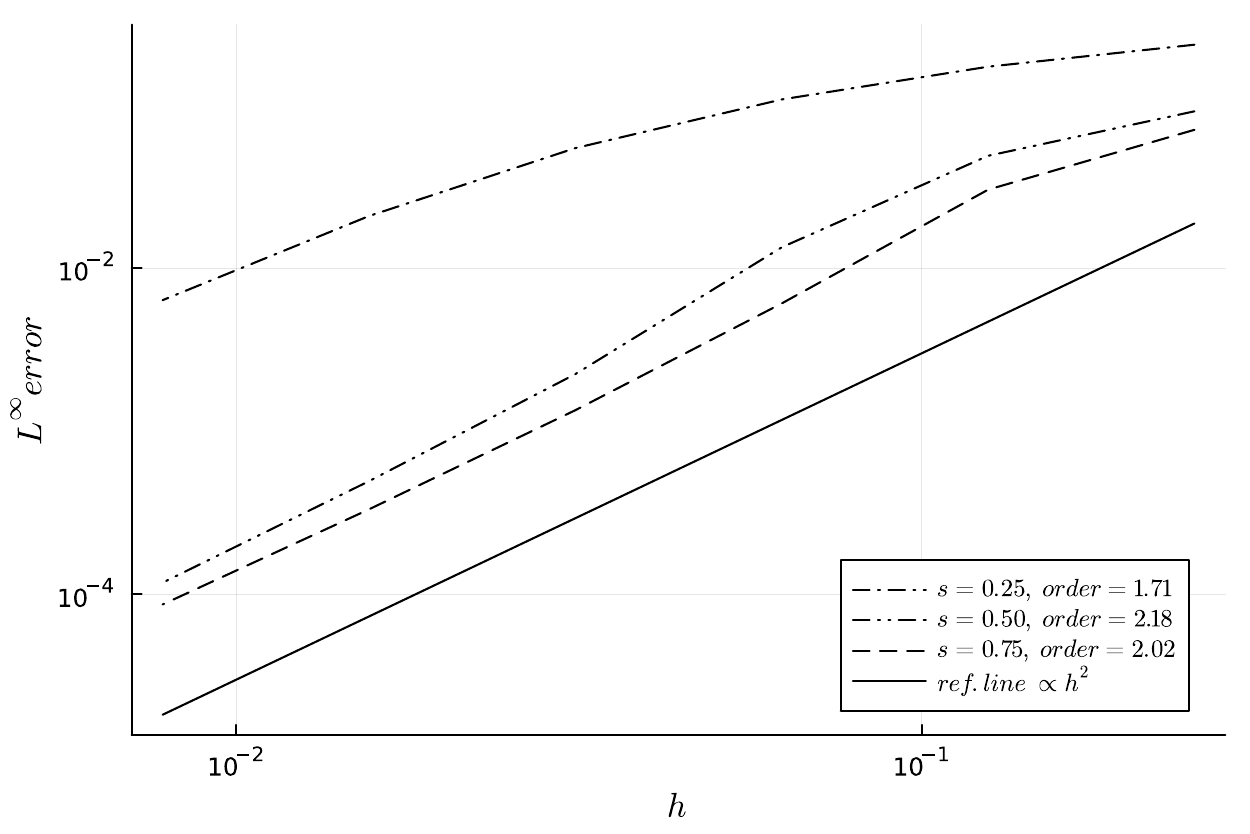}
\caption{Maximum in space and time error between numerical and analytical \eqref{eq:FracDiffExactSuperDiff} solution of the fractional diffusion equation \eqref{eq:FracDiff} with $D(x,t) = 1$ and $\alpha = 1$  computed for different discretization parameters $h$ with fixed value of $\tau=2^{-12}$.  }
\label{fig:ErrorXExactSuperDiff}
\end{figure}

\subsubsection{Fractional in space and time periodic solution}
To further investigate the performance of our scheme for $0<\alpha<1$ we can go back to an  exact solution \eqref{eq:ExactSuperDiff} and observe the following result concerning one of the eigenfunctions of the fractional Laplacian (see also \cite{barrios2019periodic}).

\begin{prop}
Let $f(x) = \sin(x)$. Then, for any $s\in [0,1]$ we have $(-\Delta )^s f(x) = f(x)$. 
\end{prop}
\begin{proof}
Notice that the fractional Laplacian can be written as a principal value integral
\begin{equation}
	(-\Delta)^s f(x) = c_s  \int_{|x-y|>0} \frac{f(x)-f(y)}{|x-y|^{1+2s}} \dd y,
\end{equation}
where $c_s^{-1} = \int_{|z|>0} (1-\cos z)|z|^{-(1+2s)} \dd z$. Therefore, by elementary trigonometric formulas
\begin{equation}
\begin{split}
	(-\Delta)^s f(x) &= c_s \int_{|x-y|>0} \frac{\sin x-\sin y}{|x-y|^{1+2s}} \dd y = 2c_s \int_{|x-y|>0} \frac{\sin \frac{x-y}{2}\cos \frac{x+y}{2}}{|x-y|^{1+2s}} \dd y \\ 
	&= 2c_s \int_{|z|>0} \frac{\sin\left( -\frac{z}{2}\right)\cos\left(x+\frac{z}{2}\right)}{|z|^{1+2s}} \dd z = 2c_s \int_{|z|>0} \frac{\sin\left( -\frac{z}{2}\right)\left(\cos x\cos\frac{z}{2}-\sin x\sin\frac{z}{2}\right)}{|z|^{1+2s}} \dd z.
\end{split}
\end{equation}
Now, we can split the integral in two parts and notice that the one multiplying $\cos x$ has an odd integrand and vanishes due to the principal value. Hence,
\begin{equation}
	(-\Delta)^s f(x) = 2c_s \sin x \int_{|z|>0} \frac{\sin^2 \frac{z}{2}}{|z|^{1+2s}} \dd z = c_s \sin x \int_{|z|>0} \frac{1-\cos z}{|z|^{1+2s}} \dd z = \sin x  = f(x),
\end{equation}
and the proof is complete. 
\end{proof}
Therefore, the solution of \eqref{eq:ExactSuperDiff} with $D(x,t) = 1$ and $u_0(x) = \sin x$ is 
\begin{equation}\label{eq:ExactSuperDiffSin}
	u(x,t) = E_\alpha(-t^\alpha)\sin x, \quad \alpha\in(0,1), \quad s\in[0,1]. 
\end{equation}
The smooth function above can be used to verify the convergence of our scheme since the Mittag-Leffler function is an eigenfunction of the Caputo derivative (\cite[formula (2.4.58)]{kilbas2006theory}). In Fig. \ref{fig:ErrorTExactSuperDiff}, we have plotted the maximum in space and pointwise in time error computed at $t = T = 1$ for various $\tau$ and $\alpha$. We have conducted a series of other calculations with different parameters $s$, $X$, $h$, and all of them produced results similar to those presented in Fig. \ref{fig:ErrorTExactSuperDiff}. In particular, we see that pointwise in time the numerical scheme achieves order $1$, that is, all error lines have a unit slope in the log-log scale. According to Corollary \ref{cor:TimeOrder} the \emph{pointwise} at $t>0$ order of our numerical scheme should be at least equal to $1-\alpha$. As we see, the numerical order is always equal to $1$ which is typical for the L1 method. For example, in \cite{kopteva2019error} this behaviour was proved for local in space subdiffusion under smoothness assumptions that were more restrictive than ours. 

\begin{figure}[h!]
	\centering
	\includegraphics[scale = 0.6]{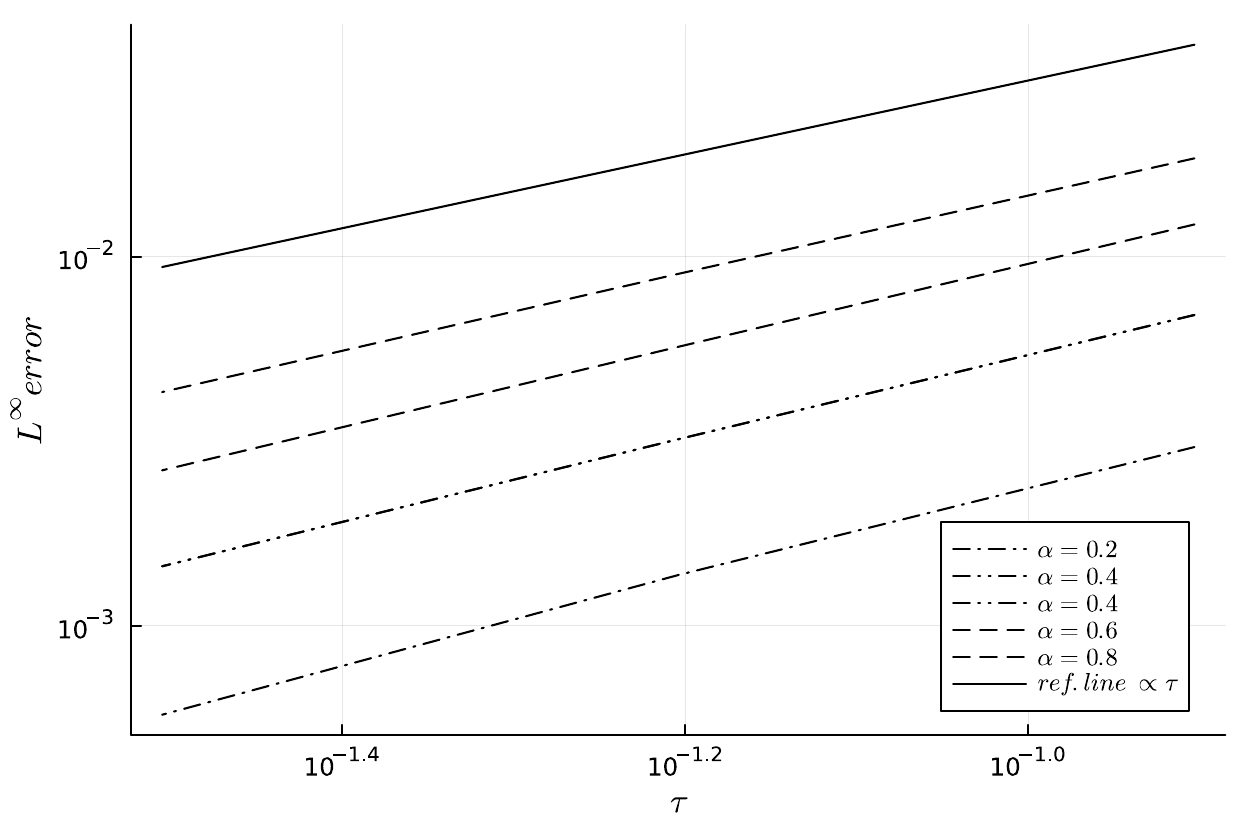}
	\caption{Maximum in space and pointwise in time error between numerical and analytical \eqref{eq:ExactSuperDiffSin} solution of the fractional diffusion equation \eqref{eq:FracDiff} with $D(x,t) = 1$ and $u_0(x) = \sin x$ computed for different discretization parameters $\tau$ with fixed value of $h=10^{-2}$. Here, $s=0.75$ and the size of domain is $X = 8\pi$. The error is at $t=1$. }
	\label{fig:ErrorTExactSuperDiff}
\end{figure}

\subsubsection{Compactly supported solution with spatially dependent diffusivity}
As a next example we construct an exact solution with compactly supported diffusivity $D$. Note that, in this case, there is no need to numerically compute the solution outside the support of $D$, since it remains constant in time. This fact is advantageous for implementation, since the domain truncation error vanishes. 

\begin{prop}
	Let the diffusivity be given by
	\begin{equation}\label{eq:DiffusivityExact}
		D(x,t) = \frac{1}{\Gamma(1+2s)}\left(1-x^2\right)^s_+.
	\end{equation}
	Then, the solution to the fractional diffusion equation \eqref{eq:FracDiff} is given by
	\begin{equation}\label{eq:FracDiffExact}
		u(x,t) = \frac{E_\alpha(-t^\alpha)}{\Gamma(1+2s)}\left(1-x^2\right)^s_+,
	\end{equation}
	where the Mittag-Leffler function $E_\alpha$ is defined by \eqref{eq:Mittag-Leffler}. The initial condition is $u_0(x) = \frac{(1-x^2)^s_+}{\Gamma(1+2s)}$. 
\end{prop}
\begin{proof}
	This is a straightforward calculation. Let $u$ be given by \eqref{eq:FracDiffExact} and recall that since the Mittag-Leffler function is an eigenfunction of the Caputo derivative we have $\partial^\alpha_t u = - u$. Moreover, by \cite[equation (6.4)]{huang2014numerical} we have
	\begin{equation}
		(-\Delta)^s \left(\frac{2^{-2s}\Gamma(\frac{1}{2})}{\Gamma(1+s)\Gamma(\frac{1}{2}+s)}(1-x^2)^s_+\right) = 1, \quad x\in (-1,1),
	\end{equation}
	and hence,
	\begin{equation}
		(-\Delta)^s u = E_\alpha(-t^\alpha) \frac{\Gamma(1+s)\Gamma(\frac{1}{2}+s)}{2^{-2s} \Gamma(\frac{1}{2})\Gamma(1+2s)}.
	\end{equation}
	The constant appearing above is actually equal to $1$, which can be seen by utilizing the duplication formula for the gamma function
	\begin{equation}
		\Gamma(1+s)\Gamma\left(\frac{1}{2}+s\right) = s \Gamma(s) \Gamma\left(\frac{1}{2}+s\right) \stackrel{\text{dupl.}}{=} s 2^{1-2s} \Gamma\left(\frac{1}{2}\right) \Gamma(2s) = 2^{-2s} \Gamma\left(\frac{1}{2}\right)\Gamma(1+2s),
	\end{equation}
	so that $D(x,t)(-\Delta)^s u = u = - \partial^\alpha_t u$. 
\end{proof}
In Fig. \ref{fig:ErrorTExact} on the left, we can see the maximum in space and time error plotted for $s=0.75$ and three different choices of $\alpha$ with varying $\tau$. The space grid parameter was fixed to be $h = 2^{-10}$ in order for the error to be dominated by the time discretization error, that is, to avoid saturation. As we can see, the numerical results are consistent with the claim that the temporal order is equal to $\alpha$. Corollary \ref{cor:TimeOrder} states that for small times this is indeed the case for $\alpha\in(0,1/2)$. Since we have numerically computed the maximal error for $t\in (0,T]$ we have been able to see the accuracy of the scheme near the origin and confirm that the order is $\alpha$. On the other hand, for $\alpha\in [1/2, 1)$ Corollary \ref{cor:TimeOrder} estimates the global order to be at least equal to $1-\alpha$ which is smaller than the numerically observed value. Note that, however, the regularity in time of the solution in this example is higher than the one assumed by us. Global order $\alpha$ for all $\alpha\in(0,1)$ is typical for L1 scheme applied to subdiffusion equations (see \cite{kopteva2019error} for the local in space problem).  

We have also conducted some other experiments that yielded the same conclusion. On the other hand, to the right of the same figure we can see the error with respect to the varying space discretization parameter $h$. The calculated order is much lower than the order of the truncation error, probably because of the low regularity of the solution and the diffusivity $D$. This case is not included in our theory, which requires that the diffusivity is at least $C^{2s+\epsilon}$. However, we wanted to present this example to show that even for diffusivities of low accuracy, the scheme converges and gives meaningful results.

\begin{figure}[h!]
	\centering
	\includegraphics[scale = 0.4]{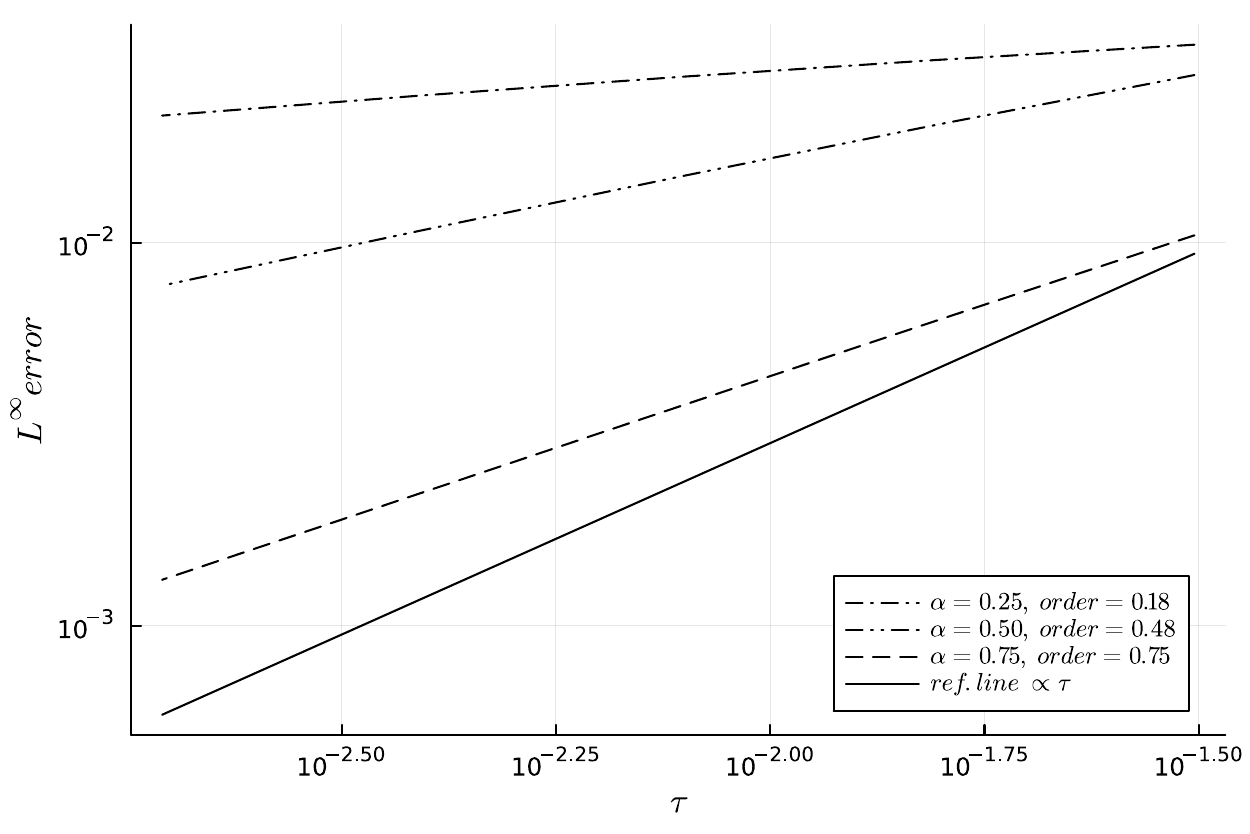}
	\includegraphics[scale = 0.4]{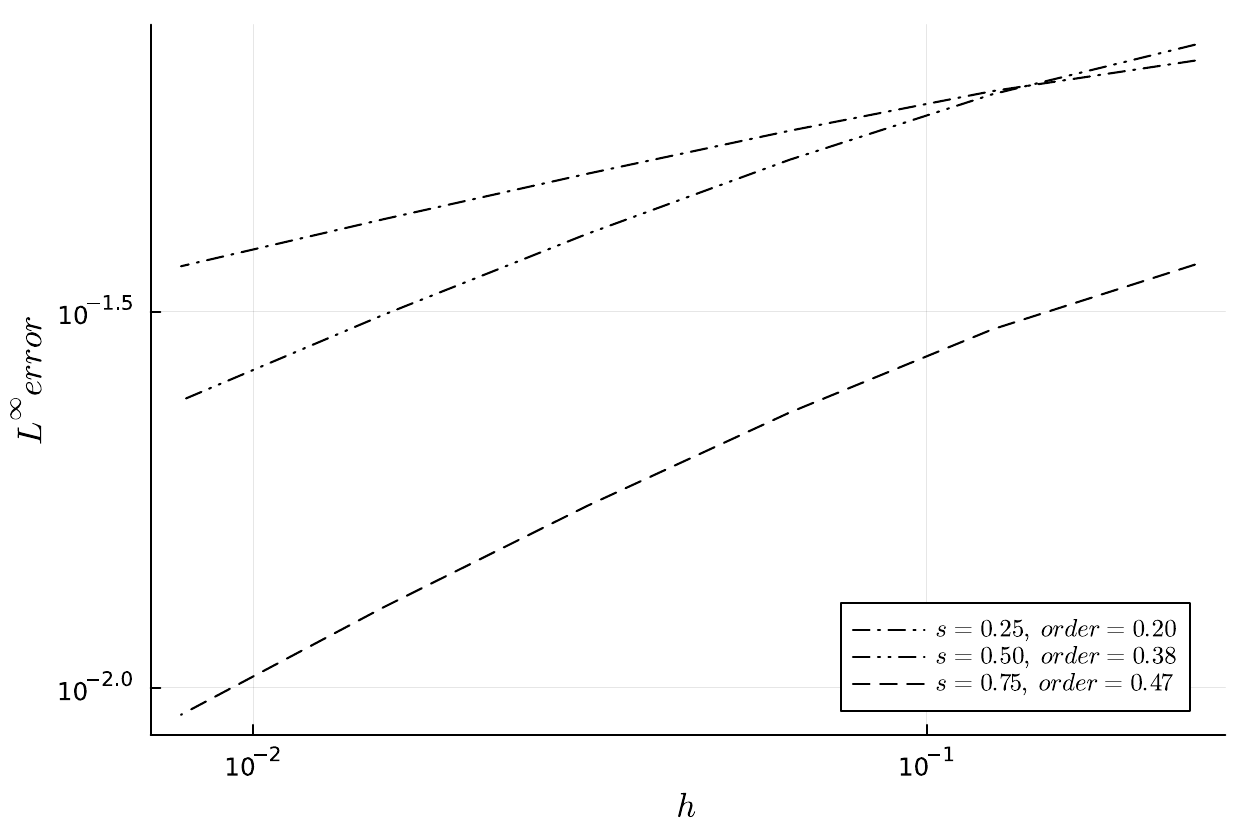}
	\caption{Maximum in space and time error between numerical and analytical \eqref{eq:FracDiffExact} solution of the fractional diffusion equation \eqref{eq:FracDiff} with diffusivity \eqref{eq:DiffusivityExact} computed for different discretization parameters: $\tau$ with fixed $h=2^{-10}$, $s = 0.75$ (left) and $h$  with fixed $\tau = 2^{-11}$, $\alpha = 0.75$ (right). }
	\label{fig:ErrorTExact}
\end{figure}

\subsection{Estimation of the order of convergence}
In the majority of situations it is impossible to obtain an exact solution to the nonlocal equation in time and space \eqref{eq:main}. However, to estimate the convergence order, we can use the extrapolation-based Aitken method (see, for example, \cite{Lin85}). Suppose that $U^n_\epsilon$ is the numerical approximation obtained with the grid parameter $\epsilon$. Then, assuming that the error behaves as $\|U^n_\epsilon - U^n_{\epsilon/2}\|_{L^\infty(\mathbb{R}\times (0,T])} \approx C \epsilon^p$ by extrapolating into half the grid we can estimate that
\begin{equation}
p \approx \log_2 \frac{\|U^n_\epsilon - U^n_{\epsilon/2}\|_{L^\infty(\mathbb{R}^d\times (0,T])}}{\|U^n_{\epsilon/2} - U^n_{\epsilon/4}\|_{L^\infty(\mathbb{R}^d\times (0,T])}}.
\end{equation}
When estimating order in time, we decrease $\tau$ while keeping $h$ fixed and \textit{vice versa}. 

\subsubsection{Fractional in space and time with $D=D(x,t)$ diffusivity}
As an example once again consider the space and time fractional diffusion with variable diffusivity and Gaussian initial condition
\begin{equation}\label{eq:DiffusivityInitial}
D(x,t) = (1+x^2+t^2)^{-1}, \quad u_0(x) = \frac{1}{\sqrt{4\pi \epsilon}} e^{-\frac{x^2}{4 \epsilon}},
\end{equation}
where $\epsilon = 0.01$. We have tested various scenarios with different grid parameters and diffusivities, and the results were mostly identical to those presented in Table \ref{tab:OrdersEst}. The first observation is that the estimated order in space is essentially equal to $2$, which coincides with the truncation error of the chosen weights \eqref{eq:LaplacianWeights}. This confirms observed fast convergence in space which justifies the choice of a moderate value of the space grid parameter $h = 2^{-6}$. On the other hand, estimation of the temporal order is more difficult to conduct numerically due to the fact that for smaller $\alpha$ the evolution of the solution proceeds more slowly. According to Corollary \ref{cor:TimeOrder} we know that the \emph{global} order in time will at least be equal to $\min\{\alpha,1-\alpha\}$.  For values of $\alpha \in (1/2, 1)$ the estimated order of convergence is consistent with $\alpha$ and increases for $\alpha \in (0,1/2)$ achieving a large value of $\approx 1.8$ for $\alpha = 0.1$.This numerical estimation may show that in reality the scheme performs better than predicted in Corollary \ref{cor:TimeOrder}. Note however, that Aitken extrapolation is only an estimation of the actual order of convergence and might need prohibitively fine grid to resolve the true order for small values of $\alpha$. However, we can conclude that again our numerical computations indicate that the scheme performs according to the previously mentioned literature findings for the L1 scheme (with better empirical results for smaller $\alpha$).

\begin{table}[h!]
\centering
\begin{tabular}{cccccc}
	\toprule
	$\alpha\backslash s$ & 0.1 &  0.25 & 0.5 & 0.75 & 0.9 \\
	\midrule 
	0.1 & 1.73 & 1.75 & 1.79 & 1.82 & 1.83 \\
	0.25 & 1.22 & 1.25 & 1.31 & 1.40 & 1.44 \\
	0.5 & 0.74 & 0.75 & 0.77 & 0.81 & 0.84 \\
	0.6 & 0.64 & 0.65 & 0.66 & 0.68 & 0.70 \\
	0.7 & 0.70 & 0.69 & 0.70 & 0.68 & 0.66 \\
	0.8 & 0.80 & 0.80 & 0.80 & 0.80 & 0.79 \\
	0.9 & 0.94 & 0.93 & 0.92 & 0.91 & 0.90 \\
	\bottomrule
\end{tabular}
\qquad
\begin{tabular}{cccccc}
	\toprule
	$\alpha\backslash s$ & 0.1 &  0.25 & 0.5 & 0.75 & 0.9 \\
	\midrule 
	0.1 & 2.03 & 2.04 & 2.02 & 2.00 & 2.00 \\
	0.25 & 2.03 & 2.04 & 2.02 & 2.00 & 2.00\\
	0.5 & 2.04 & 2.04 & 2.02 & 2.00 & 2.00 \\
	0.6 & 2.03 & 2.04 & 2.02 & 2.00 & 2.00 \\
	0.7 & 2.03 & 2.04 & 2.02 & 2.00 & 2.00 \\
	0.8 & 2.03 & 2.03 & 2.02 & 2.00 & 2.00 \\
	0.9 & 2.02 & 2.03 & 2.01 & 2.00 & 2.00 \\
	\bottomrule
\end{tabular}	
\caption{Estimated order of convergence in time (left) and space (right) for different values of $\alpha$ (rows) and $s$ (columns) based on solutions to \eqref{eq:FracDiff} with $D(x,t) = (1+x^2+t^2)^{-1}$. Here, the base values of grid parameters are $\tau = 2^{-9}$ and $h = 2^{-6}$. }
\label{tab:OrdersEst} 
\end{table}

\subsubsection{Discrete Laplacian diffusion problem}
Since our main equation \eqref{eq:main} allows for a very general class of spatial operators $\mathcal{L}^{\mu,\sigma}$, we can use this fact to test our scheme for an example that is always exact in space. That is to say, let us solve a nonlocal in time discrete in space PDE
\begin{equation}\label{eq:DiscreteLaplacian}
\partial^\alpha_t u(x,t) = D(x,t) \left(u(x+1)+u(x-1)-2u(x)\right),
\end{equation}
where the right-hand side is the central difference scheme for a unit grid spacing. Since the spatial operator is originally discrete, the numerical error of approximation vanishes, and we can closely investigate only the temporal discretization. As the diffusivity and initial condition we again take \eqref{eq:DiffusivityInitial}. The estimated order of convergence is presented in Tab. \ref{tab:OrderDiscLaplacian}. As we can see, the results are very similar to the previous case, indicating that the order is equal to $\alpha$ for $\alpha \in(1/2,1)$. For smaller values we observe a higher order approaching $1.8$. 

\begin{table}[h!]
\centering
\begin{tabular}{cccccccc}
	\toprule
	$\alpha$ & 0.1 & 0.25 & 0.5 & 0.6 & 0.7 & 0.8 & 0.9 \\
	\midrule 
	order in time
	& 1.81 & 1.32 & 0.76 & 0.66 & 0.69 & 0.80 & 0.93 \\
	\bottomrule
\end{tabular}
\caption{Estimated order of convergence in time for different values of $\alpha$ based on solutions to \eqref{eq:DiscreteLaplacian} with $D(x,t) = (1+x^2+t^2)^{-1}$. Here, the base values of grid parameters are $\tau = 2^{-9}$ and $h = 2^{-6}$. }
\label{tab:OrderDiscLaplacian}
\end{table}

\section*{Acknowledgement}
Ł.P. has been supported by the National Science Centre, Poland (NCN) under the grant Sonata Bis with a number NCN 2020/38/E/ST1/00153. F.d.T. was supported by the Spanish Government through RYC2020-029589-I, PID2021-127105NB-I00 and CEX2019-000904-S funded by the MICIN/AEI. Part of this material is based upon work supported by the Swedish Research Council under grant no. 2016-06596 while F. d. T. was in residence at Institut Mittag-Leffler in Djursholm, Sweden, during the research program "Geometric Aspects of Nonlinear Partial Differential Equations", fall of 2022.

\appendix	
\section{Discrete fractional Gr\"onwall inequalities}\label{app:Gronwall}

Here, we provide a proof of the discrete fractional Gr\"onwall inequalities for the Riemann-Liouville and Caputo L1 discretizations. The latter has previously been proved in the literature (see for ex. \cite{li2018analysis}) and below we focus mainly on the former case. However, we emphasize that our proof for the Riemann-Liouville discrete derivative is elementary, straightforward, and can be easily adapted for the Caputo case. The crucial result that we utilize is the classical version of the Gr\"onwall inequality for the discrete fractional integral.

\begin{prop}[Discrete fractional Gr\"onwall inequality (integral version) (\cite{dixon1985order}, Theorem 2.1)]
\label{prop:GronwallInt}
 Let $\alpha\in(0,1)$, $\tau>0$ and $N\in \N$. Let also $\{y^n\}_{0\leq n\leq N}$ be a nonnegative sequence satisfying 
\begin{equation}
	y^n \leq M_0 t_n^{\alpha-1} + M_1 \tau^{\alpha}\sum_{k=0}^{n-1} (n-k)^{\alpha-1} y^k + M_2, \quad \textup{for all} \quad 0\leq n\leq N,
\end{equation}
where $M_{i}$ with $i=0,1,2$ are some nonnegative constants (here  $0^{\alpha-1}=1$). Then,  we have
\[
	y^n \leq M_0\Gamma(\alpha) t_n^{\alpha-1} E_{\alpha,\alpha} (M_1 \Gamma(\alpha) t_n^\alpha) + M_2 E_\alpha(M_1 \Gamma(\alpha) t_n^\alpha), \quad \textup{for all} \quad 0\leq n\leq N,
\]
  where $E_\alpha$ and $E_{\alpha,\alpha}$ are the Mittag-Leffler functions defined in \eqref{eq:Mittag-Leffler}  and $t_n=\tau n$.

\end{prop}

\begin{lem}[Discrete fractional Gr\"onwall inequality (differential version)]\label{lem:grondif}
Let $\alpha\in(0,1)$ and $\tau>0$. Let also $\{y^n\}_{0\leq n\leq N}$ and $\{F^n\}_{0\leq n\leq N}$ be nonnegative sequences satisfying, for all $1\leq n\leq N$, the inequalities
\begin{equation}\label{eq:HypGronRL}
\partial_\tau^\alpha y^{n} \leq \lambda_0 y^n + \lambda_1 y^{n-1}+ F^n, \quad
\end{equation}
and
\begin{equation}\label{eqn:GronRLFBound}
\frac{\tau^\alpha}{\Gamma(\alpha)} \sum_{k=0}^{n-1} (n-k)^{\alpha-1}F^{k+1} \leq F_1 + t_n^{\alpha-1} F_2,
\end{equation}
where $\partial_\tau^\alpha$ is either the L1 discrete Riemann-Liouville derivative \eqref{eq:discRLder} or the L1 discrete Caputo derivative \eqref{def:disccap},  $F_1$ and $F_1$ are nonnegative constants (possibly depending on $\tau$), $t_n:=\tau n$, and $\lambda_0,\lambda_1\geq0$. Finally, let $\tau_0,M>0$ be such that 
\[
\tau_0 < (\lambda_0 \Gamma(2-\alpha))^{-\frac{1}{\alpha}} \quad \textup{and} \quad M= \frac{\Gamma(\alpha)\Gamma(2-\alpha)}{1-\lambda_0 \tau_0^\alpha\Gamma(2-\alpha)}.
\]
The following hold for all $\tau < \tau_0$ and all $0\leq n\leq N$:
\begin{enumerate}[\rm (a)]
\item \emph{(Riemann-Liouville)} If $\partial_\tau^\alpha=  {^{RL}\partial_\tau^\alpha}$, then
\begin{equation}
		y^n \leq (y_0\tau^{1-\alpha} + M F_2) \Gamma(\alpha)E_{\alpha,\alpha}(2\max\{\lambda_0,\lambda_1\} M t_n^\alpha)t_n^{\alpha-1} + M E_{\alpha} (2\max\{\lambda_0,\lambda_1\} M t_n^\alpha)F_1.
	\end{equation}
\item \emph{(Caputo)} If $\partial_\tau^\alpha=  {^{C}\partial_\tau^\alpha}$, then
\begin{equation}\label{eq:GrongCap1}
		y^n \leq M \Gamma(\alpha)E_{\alpha,\alpha}(2\max\{\lambda_0,\lambda_1\} M t_n^\alpha) F_2 t_n^{\alpha-1} + \frac{E_{\alpha} (2\max\{\lambda_0,\lambda_1\} M t_n^\alpha)}{1-\lambda_0 \tau_0^\alpha\Gamma(2-\alpha)}\left(y^0+ \Gamma(\alpha)\Gamma(2-\alpha)F_1)\right).
	\end{equation}
\end{enumerate}
\end{lem}
\begin{proof}
We mainly focus only on the Riemann-Liouville case and comment on the easier Caputo version at the end of our proof. The strategy is to ``invert" the discrete Riemann-Liouville derivative and then use Proposition \ref{prop:GronwallInt} to obtain a closed bound for $y^n$ in terms of the Mittag-Leffler function. We start by observing that the inequality (\ref{eq:HypGronRL}) is equivalent to 
\begin{equation}
	\label{eqn:Gronwall1}
	y^n \leq (b_{n-1}-b_n)y^0+ \sum^{n-1}_{k=1} (b_{n-k-1}-b_{n-k}) y^k + \mu_0 y^n + \mu_1 y^{n-1}+ \mu_2 F^n,
\end{equation}
with $\mu_0 = \lambda_0 \tau^\alpha\Gamma(2-\alpha)$, $\mu_1 = \lambda_1 \tau^\alpha\Gamma(2-\alpha)$ and $\mu_2 =  \tau^\alpha\Gamma(2-\alpha)$. 
 From here, we proceed in several steps.
 
\textbf{Step 1.} We will show that the following integral-type inequality holds for all $n\geq1$: 
\begin{equation}
	\label{eqn:Gronwall2}
	y^n \leq y^0 (n+1)^{\alpha-1} + \mu_0 \sum_{k=1}^{n} (n-k+1)^{\alpha-1} y^k+ \mu_1 \sum_{k=0}^{n-1} (n-k)^{\alpha-1} y^k + \mu_2 \sum_{k=1}^{n} (n-k+1)^{\alpha-1} F^k.
\end{equation}
We proceed by induction. First, for $n=1$ from (\ref{eqn:Gronwall1}) we have
\begin{equation}
	\begin{split}
		y^1 &\leq (b_0 - b_1) y^0 + \mu_0 y^1 + \mu_1 y^0 + \mu_2 F^1 \\
		&=   (2-2^{1-\alpha}) y^0  + \mu_0 \sum_{k=1}^{1} (1-k+1)^{\alpha-1} y^k +\mu_1 \sum_{k=0}^{0} (1-k)^{\alpha-1} y^k + \mu_2 \sum_{k=1}^{1} (1-k+1)^{\alpha-1} F^k.
	\end{split}
\end{equation}
It remains to show that $2-2^{1-\alpha} \leq 2^{\alpha -1}$ which  follows easily since 
\begin{equation}
2^{\alpha -1} +2^{1-\alpha}-2= 2^{\alpha-1} (1 +2^{2-2\alpha}-2^{2-\alpha}) =2^{\alpha-1} (2^{1-\alpha}-1)^2\geq0
\end{equation}
 Now, assume that (\ref{eqn:Gronwall2}) is valid for $i=1,\ldots, n-1$. From  \eqref{eqn:Gronwall1}, we separate the term with $y^0$, and use the induction hypothesis to get 
\begin{equation}\label{eq:ind44}
	\begin{split}
		y^n 
		&\leq (b_{n-1}-b_n)y^0 + \sum_{k=1}^{n-1} (b_{n-k-1}-b_{n-k})y^k + \mu_0 y^n + \mu_1 y^{n-1} + \mu_2 F^n\\
		&\leq y^0 \sum_{k=0}^{n-1} (b_{n-k-1}-b_{n-k}) (k+1)^{\alpha-1} \\
		&\quad + \mu_0 \left( y^{n} + \sum_{k=1}^{n-1} (b_{n-k-1}-b_{n-k})\sum_{j=1}^{k} (k-j+1)^{\alpha-1} y^j \right) \\
		&\quad + \mu_1 \left( y^{n-1} + \sum_{k=1}^{n-1} (b_{n-k-1}-b_{n-k})\sum_{j=0}^{k-1} (k-j)^{\alpha-1} y^j \right) \\
		&\quad + \mu_2 \left(F^n + \sum_{k=1}^{n-1} (b_{n-k-1}-b_{n-k})\sum_{j=1}^{k} (k-j+1)^{\alpha-1}F^j\right) \\
		&=: y^0 S_{-1} + \mu_0 S_0 + \mu_1 S_1 + \mu_2 S_2.
	\end{split}
\end{equation} 
We now proceed to estimating each of the $S_i$ terms.  We crucially rely in the following estimate
\begin{equation}
	\label{eqn:GronwallMainIneq}
	\sum_{l=1}^{n} (b_{l-1}-b_{l}) (n-l+1)^{\alpha-1} \leq (n+1)^{\alpha-1},
\end{equation}
whose proof can be found in \cite[Lemma 3.4]{jin2017analysis}. From \eqref{eqn:GronwallMainIneq}, we get the bound $S_{-1} \leq (n+1)^{\alpha-1}$ by performing the   change $l = n-k$ in the summation variable.  
 To bound $S_0$, we interchange the order of summation and  perform the change $l=n-k$ to get 
\begin{equation}
	\begin{split}
		S_0 
		&= y^{n} + \sum_{j=1}^{n-1}y^j \left(\sum_{k=j}^{n-1}(b_{n-k-1}-b_{n-k}) (k-j+1)^{\alpha-1}\right)  \\
		&= y^{n} + \sum_{j=1}^{n-1} y^j \left(\sum_{l=1}^{n-j}(b_{l-1}-b_{l}) (n-j-l+1)^{\alpha-1}\right) \\
		& \leq y^{n} + \sum_{j=1}^{n-1} (n-j+1)^{\alpha-1} y^j \\
		&= \sum_{j=1}^{n} (n-j+1)^{\alpha-1} y^j,
	\end{split}
\end{equation}
where the inequality comes from (\ref{eqn:GronwallMainIneq}) with $n$ replaced by $n-j$.   The precise same calculations show
\[
S_2\leq  \sum_{j=1}^{n} (n-j+1)^{\alpha-1} F^j.
\]

 For the last remaining sum, $S_1$, we proceed in a similar way to get
\begin{equation}
\begin{split}
S_1&=y^{n-1} + \sum_{j=0}^{n-2} y^j \left(\sum_{l=1}^{n-j-1} (b_{l-1}-b_l) ((n-l-1) -j+1)\right)\leq \sum_{j=0}^{n-1} (n-j)^{\alpha-1}y^j.
\end{split}
\end{equation}
This concludes the proof of \eqref{eqn:Gronwall2}.

\textbf{Step 2:} We prove now the statement of our result in the Riemann-Liouville case.  
Note that from (\ref{eqn:Gronwall2}) we easily get 
\begin{equation}
	\label{eqn:GronRLAux}
	(1-\mu_0)y^n \leq (1-\mu_0) y^0 (n+1)^{\alpha-1} + \sum_{k=0}^{n-1} \left(\mu_0 (n-k+1)^{\alpha-1} +\mu_1 (n-k)^{\alpha-1}\right)y^k +\mu_2(F_1+t_n^{\alpha-1}F_2),
\end{equation}
where we used the assumption \eqref{eqn:GronRLFBound} on boundedness of the  discrete  fractional integral of $F^k$. Since $\tau\leq \tau_0$ implies $0<\mu_0<1$,
\begin{equation}
	y^n \leq y^0 (n+1)^{\alpha-1} + \frac{2\max\{\mu_0,\mu_1\}}{1-\mu_0}\sum_{k=0}^{n-1}  (n-k)^{\alpha-1} y^k + \frac{\mu_2}{1-\mu_0} (F_1+t_n^{\alpha-1}F_2).
\end{equation}
Now, by unravelling definitions of $\mu_i$ we can further write 

\begin{equation}
\begin{split}
	y^n &\leq y^0 \tau^{1-\alpha} t_n^{\alpha-1} + \frac{2\Gamma(2-\alpha)\max\{\lambda_0,\lambda_1\}}{1-\lambda_0\tau_0^\alpha\Gamma(2-\alpha)} \tau^\alpha \sum_{k=0}^{n-1} (n-k)^{\alpha-1} y^k + \frac{\Gamma(\alpha)\Gamma(2-\alpha)}{1-\lambda_0\tau_0^\alpha\Gamma(2-\alpha)} (F_1+t_n^{\alpha-1}F_2)\\
	&\leq (y^0 \tau^{1-\alpha} + M F_2) t_n^{\alpha-1} + \frac{2\max\{\lambda_0,\lambda_1\} M}{\Gamma(\alpha)} \tau^\alpha \sum_{k=0}^{n-1} (n-k)^{\alpha-1} y^k + M F_1,
	\end{split}
\end{equation}
where $M$ is the one defined in the statement of the result.

Now, we can use Proposition \ref{prop:GronwallInt} in order to solve for $y^n$ which immediately brings us to the conclusion for the Riemann-Liouville case. 

\textbf{Step 3:} We prove now the statement of our result in the Caputo case. This   can be proved essentially in the same manner as in the Riemann-Liouville case. Instead of \eqref{eqn:Gronwall1}, here we have  \begin{equation}
\begin{split}
	y^n \leq b_{n-1} y^0 + \sum_{k=1}^{n-1} (b_{n-k-1}-b_{n-k}) y^k + \mu_0 y^n + \mu_1 y^{n-1} + \mu_2 F^n,
	\end{split}
\end{equation}

and from here we prove by induction that
\begin{equation}\label{eqn:GronwallCaputoClaim}
	y^n \leq y^0 + \mu_0 \sum_{k=1}^{n} (n-k+1)^{\alpha-1} y^k+ \mu_1 \sum_{k=0}^{n-1} (n-k)^{\alpha-1} y^k + \mu_2 \sum_{k=1}^{n} (n-k+1)^{\alpha-1} F^k.
\end{equation}

The case $n=1$ is trivial to check since $b_0=1$ and the rest of the terms are the same as in the Riemann-Liouville case. Thus, \eqref{eqn:GronwallCaputoClaim} is proved by induction as before, since 
\[
y^n \leq y^0 \widetilde{S}_{-1} + \mu_0 S_0 + \mu_1 S_1 + \mu_2 S_2,
\]
with $S_0$, $S_1$ and $S_2$ as in the Riemann-Liouville case, and 
\[
\widetilde{S}_{-1}= \left(b_{n-1} + \sum_{k=1}^{n-1} (b_{n-k-1}-b_{n-k})\right) y^0 = b_0y^0= y^0.
\]
From \eqref{eqn:GronwallCaputoClaim}, proceeding as before, it is standard to get

\begin{equation}
	y^n \leq \frac{M}{\Gamma(\alpha)\Gamma(2-\alpha)}y^0 +\frac{2\max\{\lambda_0,\lambda_1\} M}{\Gamma(\alpha)} \tau^\alpha \sum_{k=0}^{n-1} (n-k)^{\alpha-1} y^k + M(F_1+t_n^{\alpha-1}F_2).
\end{equation}
Invoking Proposition \ref{prop:GronwallInt} finishes the proof.  
\end{proof}
An important special case of the above inequality arises when $\lambda_{0}=\lambda_1 = 0$ and $F^n = G$ for all $n$. 
The discrete fractional integral (\ref{eqn:GronRLFBound}) of $F^n$ now becomes
\begin{equation}
\frac{\tau^\alpha}{\Gamma(\alpha)} \sum_{k=0}^{n-1} (n-k)^{\alpha-1} F^{k+1} = G\frac{\tau^\alpha}{\Gamma(\alpha)} \sum_{k=1}^{n} k^{\alpha-1} \leq G\frac{\tau^\alpha}{\Gamma(\alpha)} \int_0^n x^{\alpha-1} \dd x = G\frac{t_n^\alpha}{\Gamma(1+\alpha)},
\end{equation} 
Finally, note that, for all $l\leq n$,
\[
\frac{\tau^\alpha}{\Gamma(\alpha)} \sum_{k=0}^{l-1} (l-k)^{\alpha-1} F^{k+1} \leq G\frac{t_l^\alpha}{\Gamma(1+\alpha)}\leq G\frac{t_n^\alpha}{\Gamma(1+\alpha)}.
\]
We can immediately apply Lemma  \ref{lem:grondif} for every $n$ to obtain the following result:

\begin{cor}\label{cor:GronwallSimple}
 Let $\alpha\in(0,1)$, $\tau>0$, $N\in \N$ and $G>0$. Let also $\{y^n\}_{0\leq n\leq N}$ be a nonnegative sequence. The following hold:
\begin{enumerate}
	\item (\textit{Riemann-Liouville}) If
	$
	^{RL}\partial_\tau^\alpha y^{n} \leq G
	$
	for $1\leq n\leq N$ then
	\begin{equation}\label{eq:RLCor}
		y^n \leq y^0 \tau^{1-\alpha} t_n^{\alpha-1} + \frac{\Gamma(2-\alpha)}{\alpha} t_n^\alpha G, \quad \textup{for all}\quad 0\leq n\leq N.
	\end{equation}
	\item (\textit{Caputo}) If
	$
	^{C}\partial_\tau^\alpha y^{n} \leq G, 
	$
	for  $1\leq n\leq N$ then
	\begin{equation}\label{eq:Gronw2}
		y^n \leq y^0 + \frac{\Gamma(2-\alpha)}{\alpha}t_n^\alpha G, \quad \textup{for all}\quad  0\leq n\leq N.
	\end{equation}
\end{enumerate}
\end{cor}

\section{Technical result}
\begin{lem}\label{lem:tech}
Let $a\geq b>0$ and $\beta\in(0,1]$. Then,
\[
\beta a^{\beta-1}(a-b) \leq a^{\beta}-b^{\beta} \leq a^{\beta-1}(a-b).
\]
\end{lem}
\begin{proof}
Let us prove the right estimate first. Since $a\geq b$ then $a^{\beta-1}\leq b^{\beta-1}$ so that $a^{\beta-1}b\leq b^{\beta}$ and thus $-b^{\beta}\leq -a^{\beta-1}b$. Then
\[
a^{\beta}-b^{\beta} \leq a^{\beta}-a^{\beta-1}b\quad  \Longrightarrow \quad a^{\beta}-b^{\beta} \leq a^{\beta-1}(a-b). 
\]
The right estimate follows simply noting that
\[
\beta a^{\beta-1}(a-b) \leq \beta \int_b^a t^{\beta-1} \dd t= a^{\beta}-b^{\beta}.
\]
\end{proof}

\bibliography{biblio}
\bibliographystyle{plain}	

\end{document}